\documentclass[12pt,reqno]{amsart}
\usepackage{amsmath, amsthm, amssymb}
\usepackage[T1]{fontenc}
\usepackage[OT2, T1]{fontenc}
\usepackage[russian,english]{babel}
\usepackage{geometry}
\usepackage{cancel}
\usepackage{soul}
\usepackage{comment}
\usepackage{color}
\usepackage{amscd}
\usepackage{tabularx}
\usepackage{url}
\usepackage{eurosym}
\usepackage{braket}
\usepackage{hyperref}

\usepackage{setspace}
\usepackage{verbatim}
\usepackage{mathrsfs}
\usepackage{dsfont}
\usepackage{amsbsy}
\usepackage{amstext}
\usepackage{amsthm}
\usepackage{amssymb}
\usepackage{xcolor}
\usepackage{cancel}
\usepackage{graphicx}
\usepackage[export]{adjustbox}

\usepackage{float}
\PassOptionsToPackage{normalem}{ulem}
\usepackage[toc,page]{appendix}
\usepackage{graphicx,tikz}

\usepackage{adjustbox,graphicx,subcaption}
\makeatletter
\usepackage{eurosym}
\usepackage{mathrsfs}
\usepackage{dsfont}
\usepackage{xcolor}
\usepackage{cancel}

\newtheorem{theorem}{Theorem}[section]
\newtheorem{lemma}[theorem]{Lemma}
\newtheorem{corollary}[theorem]{Corollary}
\newtheorem{conjecture}[theorem]{Conjecture}

\newtheorem{proposition}[theorem]{Proposition}
\newtheorem{problem}[theorem]{Problem}

\theoremstyle{definition}
\newtheorem{definition}[theorem]{Definition}

\newtheorem{remark}[theorem]{Remark}

\numberwithin{equation}{section}

\def \bN {\mathbb N}

\def \bQ {\mathbb Q}
\def \bR {\mathbb R}

\def \bZ {\mathbb Z}

\def\bfa{{\mathbf a}}

\def\bfc{{\mathbf c}}
 
\def\bff{{\mathbf f}}

\def\bfh{{\mathbf h}}
\def\bfj{{\mathbf j}}

\def\bfp{{\mathbf p}}

\def\bft{{\mathbf t}}

\def\bfv{{\mathbf v}}

\def\bfx{{\mathbf x}}
\def\bfy{{\mathbf y}}
\def\bfz{{\mathbf z}}

\def\bfF{{\mathbf F}}

\def \sA {\mathscr A}
\def \sB {\mathscr B}

\def \sU {\mathscr U}

\def \sX {\mathscr X}

\def \rd {\mathrm d}

\def\cH{{\mathcal H}}

\def\calM{{\mathcal M}}

\def \fB {\mathfrak B}
\def \fC {\mathfrak C}

\def \fM {\mathfrak M}

\def \fS {\mathfrak S}

\def \cH {\mathcal H}
\def \cI {\mathcal I}
\def \cJ {\mathcal J}

\def \cM {\mathcal M}

\def \cR {\mathcal R}
\def \cS {\mathcal S}

\def \cU {\mathcal U}
\def \cV {\mathcal V}

\def \bzero {\mathbf 0}

\def \balp {{\boldsymbol{\alp}}}

\def\C{{\mathbb C}}\def\N{{\mathbb N}}
\def\R{{\mathbb R}}
\def\Z{{\mathbb Z}}\def\Q{{\mathbb Q}}

\def\bump{{\textcolor{black} {w}}}
\def \supp {{\mathrm{supp}}}

\def \ds1 {\mathds{1}}

\def\alp{{\alpha}}

\def\del{{\delta}}

\def\ome{{\omega}} 
\def\d{{\partial}}

\def\implies{\Rightarrow}

\def\le{\leqslant} \def\ge{\geqslant}

\def\d{{\,{\rm d}}}

\def\d{{\mathrm{d}}}

\def\wgood{{\Omega_{\bfp, \mathrm{good} }}}
\def\wbad{{\Omega_{\bfp, \mathrm{sub} }}}

\DeclareMathOperator{\dist}{dist}
\DeclareMathOperator{\codimension}{codim}

\newcounter{@ToDo}
\newcommand{\todo@helper}[1]{%
	({\color{blue}TODO~\arabic{@ToDo}: {#1\@addpunct{.}}})%
}
\newcommand{\todo}[1]{\stepcounter{@ToDo}%
	\relax\ifmmode\text{\todo@helper{#1}}%
	\else\todo@helper{#1}\fi%
}
\newcounter{@cdo}
\newcommand{\cdo@helper}[1]{%
	({\color{red}CITE~\arabic{@cdo}: {#1\@addpunct{.}}})%
}
\newcommand{\cdo}[1]{\stepcounter{@cdo}%
	\relax\ifmmode\text{\cdo@helper{#1}}%
	\else\cdo@helper{#1}\fi%
}

\author{Damaris Schindler}
\author{Rajula Srivastava}
\author{Niclas Technau}

\begin{document}


\address{Damaris Schindler; 
\newline G\"{o}ttingen University, Bunsenstrasse 3--5, 37073 G\"{o}ttingen, Germany.}
\email{damaris.schindler@mathematik.uni-goettingen.de}

\address{Rajula Srivastava; 
\newline Mathematical Institute, University of Bonn, Endenicher Allee 60,
53115, Bonn, Germany, and
\newline Max Planck Institute for Mathematics, Vivatsgasse 7,
53111, Bonn,
Germany.}
\email{rajulas@math.uni-bonn.de}

\address{Niclas Technau; 
\newline 
Graz University of Technology, 
Institute of Analysis and Number Theory, Steyrergasse 30/II, 8010 Graz, Austria}
\email{ntechnau@tugraz.at}

\title[Rationals Near Manifolds, Dynamics, and Oscillatory Integrals]
{Rational Points Near Manifolds, Homogeneous Dynamics,
and Oscillatory Integrals}
\date{\today}

\begin{abstract}
Let $\cM\subset \bR^n$ be a 
compact and sufficiently 
smooth manifold of dimension $d$.
Suppose $\cM$ is nowhere completely flat.
Let $N_\cM(\delta,Q)$ denote 
the number of rational vectors $\bfa/q$
within a distance of $\delta/q$ 
from $\cM$ so that $q \in [Q,2Q)$. 
We develop a novel method to analyse $N_\cM(\delta,Q)$.
The salient feature of our technique is the combination of powerful quantitative non-divergence estimates, in a form 
due to Bernik, Kleinbock, and Margulis \cite{BKM},
with Fourier analytic tools. The second ingredient enables us to eschew the Dani correspondence and an explicit use of the geometry of numbers.

We employ this new method to address in a strong sense
a problem of Beresnevich \cite{Bers12}
regarding lower bounds on $N_\cM(\delta,Q)$
for non-analytic manifolds. 
Additionally, 
we obtain asymptotic formulae 
which are the first of their kind for such a general class of manifolds.
As a by-product, we improve upon upper bounds 
on $N_\cM(\delta,Q)$ from
a recent breakthrough of Beresnevich and Yang \cite{BY} and recover their convergence Khintchine type theorem for arbitrary nondegenerate submanifolds. Moreover, we obtain new Hausdorff dimension and measure refinements for the set of well-approximable points for a range of Diophantine exponents close to $1/n$. 
\end{abstract}

\maketitle

\section{Introduction}
The main goal of this article is 
to quantitatively study the distribution
of rational points close to manifolds. 
Given a compact manifold $\cM\subset \bR^n$, 
as one approach to measure the density of 
rational points close to $\cM$, 
we may ask for the number of 
rational points with bounded denominator, 
which lie close to $\cM$. 
To make this more precise, 
we define the following counting function. 
Let $Q\in \bN$, $0<\delta< 1/2$, 
and let $\dist$ denote a fixed distance 
function on $\bR^n$. Define the counting function
$$N_\cM(\delta, Q)=\sharp\{\mathbf{p}/q\in \mathbb{Q}^n:\ 1\leq q\leq Q,\ \mathbf{p}\in \bZ^n,\  \dist(\mathbf{p}/q,\cM)\leq \delta/q   \}.$$
Understanding the growth of the counting function $N_\cM(\delta, Q)$ in terms of asymptotics, upper and lower bounds, is a fundamental question which is of central importance in different fields such as Diophantine approximation and Diophantine geometry (see \S \ref{subsec intro app} for a more detailed discussion).\par
In this article we present a novel approach towards estimating the counting function $N_\cM(\delta, Q)$, which combines Fourier analytic tools with a deep result of Bernik, Kleinbock and Margulis \cite{BKM}. With this we recover the main theorem of a recent breakthrough by Beresnevich and Yang \cite{BY} on  a convergence Khintchine type theorem for arbitrary nondegenerate submanifolds of $\mathbb{R}^n$. Moreover, we improve on their results in terms of Hausdorff $s$-measure refinements and with this we obtain the exact value of the Hausdorff dimension of $\tau$-well approximable points lying on a nondegenerate submanifold for a larger range of values of $\tau$. In contrast to \cite{BY} we even find asymptotic formulas for the counting function $N_\cM(\delta, Q)$, as soon as $\delta$ is sufficiently large and the manifold $\cM$ nondegenerate.\par
In another direction, we obtain lower bounds of the expected order of magnitude for $N_\cM(\delta, Q)$ for smooth nondegenerate submanifolds as soon as $\delta \gg Q^{-\frac{3}{2n-1} + \varepsilon}$. This improves on seminal work of Beresnevich \cite{Bers12} in those cases where the codimension $m$ of the submanifold satisfies $m>\frac{2n-1}{3}$ and where the dimension of the submanifold is at least $2$. Moreover, Beresnevich \cite{Bers12} needs to assume the submanifold to be analytic, whereas a certain level of smoothness is sufficient for our strategy.\par

In the following we work with $l$-nondegenerate manifolds for some natural number $l$, where we use a definition of nondegeneracy as in work of Kleinbock and Margulis \cite{KM 1998}.
\begin{definition}
We say that a $d$-dimensional manifold, which is locally given by an $l$-times continuously differentiable map $\mathbf{f}: \mathbb{R}^d 
\supset \sU \rightarrow \R^n$, with $\sU$ open, is $l$-nondegenerate at a point $\mathbf{x}_0\in \sU$ if the partial derivatives of $\mathbf{f}$ of order up to $l$ in the point $\mathbf{x}_0$ span $\R^n$. Moreover, we say that the manifold is $l$-nondegenerate, if it is $l$-nondegenerate at almost every point $\mathbf{x}_0\in \sU$, with respect to the $d$-dimensional Lebesgue measure.
\end{definition}

Given the counting function $N_\cM(\delta, Q)$, one may first ask the basic question of its growth as $Q$ and $\delta$ vary. If $\cM$ is of dimension $d$, then a first upper bound is given by
$$N_\cM(\delta, Q)\ll Q^{d+1},$$
which is in fact sharp for the case where $\cM$ is a piece of a rational linear subspace of $\bR^n$. If we write $m=n-d$ for the codimension of $\cM$, then a probabilistic heuristic suggests the growth order
\begin{equation}\label{eqnheuristic}
N_\cM(\delta, Q) \asymp \delta^m Q^{d+1}.
\end{equation}
We cannot expect this heuristic to hold for all manifolds or all values of $\delta$, as examples like linear subspaces or the Fermat curve or any manifold given as a piece of a variety with many rational points show. For a more detailed discussion on this, see for example \cite{Bers12} or \cite{Huang20}. However, if we assume in some sense that $\cM$ is not flat and $\delta$ is not too small, then one may conjecture that the probabilistic expectation in \eqref{eqnheuristic} gives the correct order of magnitude. In \cite[Conjecture 3.1]{HuangE} Huang has posed the following main conjecture.

\begin{conjecture}
\label{conj Huang}
Let $\cM$ be a compact $l$-nondegenerate submanifold in $\bR^n$ of dimension $d$ and codimension $m$ with $l=m+1$. If $\delta \geq Q^{-\frac{1}{m}+\epsilon}$ for some $\epsilon >0$ and $Q\rightarrow \infty$, then
$$N_\cM (\delta, Q) \sim c_{\cM} \delta^m Q^{d+1},$$
where $c_{\cM}$ is a constant only depending on $\cM$.
\end{conjecture}

In particular, this conjecture implies the upper bound
$$N_\cM(\delta, Q) \ll \delta^m Q^{d+1} + Q^{d+\epsilon},$$
which is conjectured to hold for all $\delta$ and $Q$ as soon as $\cM$ satisfies suitable nondegeneracy conditions.\par

Establishing bounds and asymptotics for the counting function $N_\cM(\delta,Q)$ has been an active area of research in the last couple of decades. First progress has been achieved for the case of planar curves with non-vanishing Gaussian curvature, a near optimal upper bound for $C^2$ curves was established by Huxley \cite{Hux94}, the expected upper bound for $C^3$ curves by Vaughan and Velani \cite{Vaughan-Velani-2007}, a lower bound of the expected order of magnitude for $\del\gg Q^{-1}$ and $C^3$ curves by Beresnevich, Dickinson and Velani \cite{BDV07}, and the main conjecture for $C^3$ curves by Huang \cite{HuangRPAdv}.\par
For the case of general manifolds, 
Beresnevich obtained in his seminal article \cite{Bers12} the expected lower bound 
for $N_\cM(\delta, Q)$ where $\del \gg Q^{-\frac{1}{m}}$, assuming that $\cM$ is 
an analytic submanifold nondegenerate in one point. An upper bound, which is 
close to the one predicted by the main conjecture, has been obtained by Beresnevich, 
Vaughan, Velani and Zorin \cite{BVVZ17} for hypersurfaces of nonzero Gaussian curvature. Finally, using an approach based on harmonic analysis and an elegant duality principle, 
J.-J. Huang \cite{Huang20} managed to show the main conjecture for smooth compact hypersurfaces 
with Gaussian curvature bounded away from zero. 
The ideas of Huang have been further developed 
and used for manifolds of higher codimension 
under suitable curvature conditions by 
the first author and Yamagishi \cite{SY22}, 
and for hypersurfaces parametrised 
by homogeneous functions with 
non-vanishing curvatures away from 
the origin by the second 
and third author \cite{ST}.\par
In the case of smooth nondegenerate manifolds, Beresnevich and Yang \cite[Thorem 1.3]{BY} 
establish upper bounds of the expected 
order of magnitude outside of a small set
of the manifold. It is this line of attack 
that we combine with Fourier analytic methods 
and additional lattice point counting arguments, 
and that we push to upper bounds 
for any smooth nondegenerate manifold, 
and moreover asymptotics 
and lower bounds for certain ranges of 
$\delta$ and $Q$. 

For a lucid exposition and 
to state our results, 
we use the following (usual) conventions.\\
\\
{\emph{Notation:}}
Let $\sX\subset \R^d$, and let 
$f,g:\sX \rightarrow \C$ be functions.
The Vinogradov notation $f \ll g$ means that there is some constant
$C>0$ so that the functions 
satisfy $\vert f(\bfx) \vert \leq C \vert g(\bfx)\vert$ 
for all $\bfx\in \sX$. In the following 
$\Vert x \Vert:=\mathrm{dist}(x,\bZ)$ 
denotes the distance of $x$ to the nearest integer
and should not be confused with an euclidean $p$-norm
of $\bfx \in \R^d$ which we denote by
$\Vert \bfx \Vert_p$, whenever 
$1\leq p\leq \infty$. Unless otherwise specified, 
the components of a vector, say, $\bfx\in \R^d$ 
shall be given by an appropriate subscript, i.e. $\bfx$
has the components $x_1,\ldots,x_d$. Furthermore, 
we use multi-index notation in the form 
that for $\balp \in \R^d_{>0}$, we define
$\bfx^{\balp}:=x_1^{\alpha_1}\ldots x_d^{\alpha_d}$.
The Lebesgue measure on $\R^k$ is abbreviated by $\mu_k$ for $k\in \mathbb{N}$ and we write $\langle \cdot, \cdot \rangle$ for the standard scalar product on $\R^k$. Throughout, the open ball (in $\R^d$) with radius $r$
and centre $\bfc$ is denoted by 
$$
\sB(\bfc,r):=
\{\bfx \in \R^d: \Vert \bfx - \bfc \Vert_2 < r \}.
$$
The unit ball $\sB(\bzero,1)$ shall be denoted by $\sU_d$. Fix Schwartz functions 
$\Omega: \bR^d \rightarrow [0,1]$,
$\ome: \bR \rightarrow [0,1]$, and
$w: \bR \rightarrow [0,1]$. 
To avoid technical nuisances, 
we call a triple $\bft:=(\Omega, \omega, w)$ 
of such functions an
{\emph{admissible weight tuple}}, if
\begin{itemize}
\item the support of $\Omega$ is contained in $\sB(\mathbf{0},1/2)$ and
    \item 
    the support of $\ome$ is contained in $[1/2,1]$ and
    \item the support of
    $w$ is contained in 
    $(-1,1)$ and $w(x)=w(-x)$ for all $x\in \mathbb{R}$.

\end{itemize}
For a Lebesgue measurable set $A\subseteq \mathbb{R}^n$, $\dim (A)$ and $\cH^s (A)$ shall denote the Hausdorff dimension and the $s$--dimensional Hausdorff measure, respectively.

\subsection{Main Results}
Let $\mathcal{M} \subset \mathbb{R}^n$, where $n\geq 2$, 
be a bounded submanifold. For the questions above, 
it is sufficient to work locally. 
Hence by the Implicit Function Theorem 
(and possibly translating and re-scaling), 
there is no loss of generality in assuming that
$\mathcal{M}$ is in a normalised Monge form, that is
\begin{equation}\label{eq normal Monge}
\mathcal{M}=\{(\bfx,\bfF(\bfx)), \bfx\in 
\sU_d\}.    
\end{equation}
We stress that the function $$
\bfF: \sU_d \rightarrow  \mathbb{R}^m, 
\quad \bfF(\bfx)=(F_1(\bfx),\ldots, F_m(\bfx))
$$ is frequently used throughout the manuscript. 
We say that $\mathbf{F}$ is $l$-nondegenerate 
if the manifold $\calM$, as given above, is $l$-nondegenerate.\par
Let $Q\in \mathbb{N}$, $\del\in (0,1/2)$ and $\bft:=(\Omega,\omega,w)$ be an admissible weight tuple.
In this manuscript, we shall 
bound the number of rational points close to $\cM$
by studying the smooth counting function
\begin{equation}\label{def counting func}
N_\bfF^{\bft}(\delta,Q) :=
\sum_{(q,\bfa)\in \Z^{d+1}} 
\Omega\left(\frac{\bfa}{q}\right) 
\ome\left(\frac{q}{Q}\right)
\prod_{1\leq i\leq m} 
w\bigg(
\frac{\|q F_i (\frac{\bfa}{q}) \|}{\delta}
\bigg).
\end{equation}

The following statements are our main results. 
\begin{theorem}[Asymptotics]\label{thm main asymp bounds}
Let $0<\eta\leq \frac{1}{2n+10}$. Suppose $\bfF$ is $l$-nondegenerate, and at least 
$\lceil \frac{n+2}{\eta}\rceil$ times continuously differentiable. Assume that $Q$ is sufficiently large, depending on $\bfF$.
For an admissible weight tuple
$\bft =(\Omega,\omega,w)$, define the constant
\begin{equation}\label{def leading constant}
c_{\bft} 
:= \int_{\R^d} \Omega(\bfx) \rd \bfx
\cdot 
\int_{\R} y^d\omega(y) \rd y
\cdot \Big(\int_\R w(z) \rd z\Big)^m.
\end{equation}
Assume that
\begin{equation}\label{range delta}
Q^{\frac{-3+(2n+10)\eta}{2md(2l-1)(n+1)+2n-1}} \leq \delta \leq \frac{1}{2}.
\end{equation}
Then 
$$
N_{\bfF}^{\bft}(\delta,Q)
 =c_\bft \delta^m Q^{d+1} +O\left( Q^{d+1}(\delta^{-(n-\frac{1}{2})}
Q^{-\frac{3}{2}+ (n+5)\eta} )^{\frac{1}{d(2l-1)(n+1)}} \right), $$
where the implied constant is uniform 
for $\delta$ in the range \eqref{range delta}.\par
In particular, if we assume that 
$$ 
Q^{\frac{-3+(2n+10)\eta}{2md(2l-1)(n+1)+2n-1}}
Q^{\frac{d(2l-1)(n+1)\epsilon}{ (n-1/2)}}\leq 
\delta \leq \frac{1}{2},\qquad 
0<\eta \leq \frac{3}{2(n+5)}
$$
for some $\epsilon>0$, then
$$
N_{\bfF}^{\bft}(\delta,Q) =
[c_\bft + O(Q^{-\epsilon})] 
\delta^m Q^{d+1}.
$$
\end{theorem}
The main novelty of this theorem is that we only need to assume smoothness and nondegeneracy, but no condition on Gaussion curvature is needed. 
In comparison to Conjecture \ref{conj Huang}, we only obtain here results for a smaller range of $\delta$, but in this range we can establish the expected order of magnitude which matches the main conjecture.\par
One may also compare this to work of Srivastava and Technau \cite{ST} on hypersurfaces parametrized by a homogeneous function of degree at least two with non-vanishing curvature away from the origin, where for the first time they are able to obtain essentially sharp estimates for $N_{\bfF}^{\bft}(\delta,Q)$ for $\delta \geq Q^{-1+\varepsilon}$.  As discussed in \cite{Bers12, Huang20, ST}, we cannot expect the asymptotic formula to hold for $\delta$ as small as $Q^{-\frac{1}{m}+\epsilon}$ under the very weak condition of nondegeneracy. It is an interesting problem to determine the right conditions on $\delta$ for the asymptotic to be true for the class of $l$-nondegenerate smooth manifolds.\par

Further, we show lower bounds. 
\begin{theorem}[Lower Bounds]\label{thm main lower bounds}
Let $0<\eta\leq \frac{1}{8}$ and $a_n:=\frac{2n+12}{2n-1}$.
Suppose $\bfF$ is $l$-nondegenerate, and at least 
$\lceil \frac{n+1}{\eta}\rceil$ times continuously differentiable. 
If $\delta \in 
(Q^{-\frac{3}{2n-1}+a_n\eta},1/2)$, then 
$$
N_{\bfF}^{\bft}(\delta,Q)
    \geq    
c_{\bft} \delta^{m} Q^{d+1} +O\left(\delta^m Q^{d+1-\frac{\eta}{d(2l-1)(n+1)}}\right).
$$
\end{theorem}

Note that lower bounds of the expected order of magnitude have been established by Beresnevich \cite{Bers12} for nondegenerate analytic manifolds as soon as $\delta \gg Q^{-\frac{1}{m}}$, and for analytic curves with a certain mild condition on a Wronskian, comparable to nondegeneracy, for $\delta \gg Q^{-\frac{3}{2n-1}}$. In particular, in the case of curves, the aforementioned result is stronger than our result by an $\varepsilon$, since the latter allows us to take $\delta \gg Q^{-\frac{3}{2n-1}+\varepsilon}$ for any $\varepsilon >0$ for smooth nondegenerate curves. For general manifolds, we obtain a strictly larger range for $\delta$ if we have the lower bound $m> \frac{2n-1}{3}$ on the codimension $m$. Moreover, a new aspect of our theorem is that we only need to assume the nondegenerate manifold to be smooth, as opposed to analytic in \cite{Bers12}. With this we also answer a question of Beresnevich and Kleinbock in \cite[Section 5.1]{BK22}.\par

The next theorem establishes upper bounds for $N_{\bfF}^{\bft}(\delta,Q)$.
\begin{theorem}[Upper Bounds]\label{thm main upper bounds}
Let $0<\eta \leq \frac{1}{2n+10}$. Suppose $\bfF$ is $l$-nondegenerate, and at least 
$\lceil \frac{n+1}{\eta}\rceil$ times continuously differentiable. 
 Then for $0\leq \delta < 1/2$ we have
 $$N_{\bfF}^{\bft}(\delta,Q)\ll \delta^m Q^{d+1} + Q^{d+1- \frac{3m-m(2n+10)\eta}{2md(2l-1)(n+1)+2n-1}}   .$$
\end{theorem}
Note that if 
\begin{equation}\label{eqndelta3}
\delta \geq Q^{\frac{-3+(2n+10)\eta}{2md(2l-1)(n+1)+2n-1}},
\end{equation}
then it is the first term which dominates, yielding the exact upper bound as that obtained in \cite[Theorem 1.3]{BY}. This agrees with the heuristic expectation. In \cite{BY}, the authors formulate their theorem in a way, that they give an upper bound which excludes a small set of the manifold, which they call special parts. We combine our estimates on the `good set', also called generic parts in \cite{BY}, with bounds obtained by lattice considerations for the `sub-level set', to derive an upper bound for the whole counting function $ N_{\bfF}^{\bft}(\delta,Q)$. Indeed, the second term in the upper bound above corresponds to the contribution from the special parts in the language of \cite{BY}.\par

A nice feature of the Theorems \ref{thm main asymp bounds}-\ref{thm main upper bounds}, is that we can explicitly keep track of the smoothness (of $\mathbf{F}$) that is needed to establish the asymptotics, upper and lower bounds.\par 
Next, we state a consequence of Theorem \ref{thm main upper bounds} which provides an upper bound for the number of rational points {\em on} the manifold $\cM$.

\begin{corollary}\label{cor upper bound}
Suppose $\bfF$ is $l$-nondegenerate, and at least 
$2(n+1)(n+5)$ times continuously differentiable.
Then we have the upper bound
$$N_{\bfF}^{\bft}(0,Q) \ll  Q^{d+1- \frac{1}{2ld(n+1)}}.$$

\end{corollary}
First, we observe that the bound in Corollary \ref{cor upper bound} improves upon the trivial upper bound $Q^{d+1}$. Its strength lies in the generality to which it applies and in the simple form it takes. In particular, it also applies to pieces of projective varieties, that can in Monge form be represented by an $l$-nondegenerate function. For normal projective varieties, Tanimoto has recently given a very general upper bound for the number of rational points of bounded height, which is governed by a certain ``$\delta$ invariant'' that he introduces in \cite{Tanimoto}. His methods are of geometric nature and in particular require geometric understanding in the computation of the $\delta$ invariant. Our bounds show that we can already get {\em some} non-trivial upper bound by completely forgetting the specific algebraic-geometric structure of $\cM$, and working purely analytically.\par
In the case that $\cM$ is a piece of a projective variety (after considering an affine patch of it), the counting function $N_{\bfF}^{\bft}(0,Q)$ counts the number of rational points with bounded denominator that lie on a given piece of the variety. This is linked to deep problems in Diophantine geometry. For Fano varieties we have the Batyrev-Manin conjecture \cite{BM90} for the number of rational points of bounded height on such varieties. If we consider more general classes of projective varieties defined over $\bQ$, we have Serre's dimension growth conjecture, which gives, under the very weak assumption that the variety is irreducible and of degree at least $2$, non-trivial upper bounds for $N_{\bfF}^{\bft}(0,Q)$. For a precise formulation of Serre's dimension growth conjecture see \cite{Serre} and for work towards the conjecture as well as the solution of a weak form of it by works of Browning, Heath-Brown and Salberger see \cite{Browning}, \cite{BHS06}, \cite{Sal07} as well as references therein.\par
In \cite{Huang20} Huang has reproved Serre's dimension growth conjecture for certain classes of varieties, for which his condition on nonvanishing Gaussian curvature is satisfied,
and extended it to a natural class 
of manifolds. In our upper bound we do not recover the full strength of the dimension growth conjecture, but under the much milder restriction of nondegeneracy, we show that we can already get nontrivial upper bounds.\par

\subsection{Applications to Diophantine approximation}
\label{subsec intro app}
In a different direction, upper and lower bounds for $N_{\bfF}^{\bft}(\delta,Q)$ have important applications in Diophantine approximation on manifolds.

\begin{definition}
    \label{def psi approx}
    Given a function $\psi:(0,+\infty)\to(0,1)$,  we call a point $\bfy\in\R^n$ {\em $\psi$-approximable} if the condition
\begin{equation}\label{psi-app}
 \left\|\bfy-\frac{\bfa}{q}\right\|_{\infty}<\frac{\psi(q)}{q}
\end{equation}
holds for infinitely many $(\bfa,q)\in\Z^n\times\N$. 
\end{definition}
We shall denote the set of $\psi$-approximable points in $\R^n$ by $\cS_n(\psi)$. An approximation function of particular importance is the map $\psi_{\tau}(q)=q^{-\tau}$ for some $\tau>0$. To abbreviate notation, we will in this case write $\cS_n(\tau):=\cS_n(\psi_{\tau})$, and call it simply the set of $\tau$-approximable points. By Dirichlet's theorem \cite{Schmidt-1980}, $\cS_n(1/n)=\R^n$. Indeed, the classical Khintchine’s theorem states that for any decreasing approximating function $\psi$,
$$\sum_{q=1}^{\infty}\psi(q)^n<\infty \implies \mu_n\left(\cS_n(\psi)\right)=0,\qquad\, \textbf{(Convergence Case)}$$
while
$$\sum_{q=1}^{\infty}\psi(q)^n=\infty \implies \mu_n\left(\mathbb{R}^n\setminus\cS_n(\psi)\right)=0. \qquad\, \textbf{(Divergence Case)}$$

Hausdorff dimension and Hausdorff measure refinements of the above result are also well known. For a Lebesgue measurable set $A\subseteq \mathbb{R}^n$, let $\dim (A)$ denote its Hausdorff dimension, and $\cH^s (A)$ its $s$--dimensional Hausdorff measure. A classical result due to Jarn\'ik and Besicovitch \cite{Besicovich-1934, Jarnik-1929} says that, for $\tau \geq 1/n$, 
\begin{equation}
  \label{eq JaBe}
  \dim \cS_n(\tau)=\frac{n+1}{\tau+1}\,.
\end{equation}

Another fundamental theorem due to Jarn\'ik \cite{Jarnik-1931} states that given a monotonic function $\psi$ and $s\in (0, n)$,
\begin{equation*}
  \cH^s\big(\cS_n(\psi)\big)=\left\{\begin{array}{cc}
                                  0 & \text{if }\sum_{q=1}^\infty \
 q^n\Big(\frac{\psi(q)}q\Big)^{s}<\infty\,, \\[2ex]
                                  \infty & \text{if }\sum_{q=1}^\infty \
 q^n\Big(\frac{\psi(q)}q\Big)^{s}=\infty\,.
                                \end{array}
  \right.
\end{equation*}
\smallskip
On the other hand, establishing the above results for $\psi$-approximable set of points on a generic \textit{manifold}, has proven to be a notoriously hard problem. The divergence case in Khintchine's theorem, for nondegenerate $C^3$ curves in the plane, was established in \cite{BDV07}, before being completely settled for nondegenerate \textit{analytic} manifolds of arbitrary dimension (and codimension) in \cite{Bers12}. The convergence case is even more challenging. The seminal work \cite{Vaughan-Velani-2007} answered this question for non-degenerate $C^3$ planar curves. This result was later extended to weakly nondegenerate planar curves in \cite{HuangRPAdv}, see also \cite{BZ} for earlier work into that direction. 

The main result of \cite{BY} was a breakthrough in that it completely settles the convergence case of Khintchine's theorem for manifolds of arbitrary dimension (and codimension), under just a nondegeneracy condition. More precisely, \cite[Theorem 1.2]{BY} establishes that if a monotonic approximating function $\psi$ satisfies the condition
$$\sum_{q=1}^\infty \psi (q)^n<\infty, $$
then for any $d$- dimensional nondegenerate submanifold $\cM\subseteq \mathbb{R}^n$, $\mu_d$ almost all points on $\cM$ are not $\psi$-- approximable.
For comparison, the best known results before \cite{BY} were able to settle the convergence question for manifolds only under very rigid geometric and dimensional constraints, see \cite[p. 3]{BY}.

The authors remark in \cite{BY} that these assumptions were often a by-product of using Fourier analytic tools. In contrast, the approach in \cite{BY} relied on combining techniques from homogeneous dynamics and the geometry of numbers with quantitative non-divergence estimates from \cite{BKM}. One of the primary goals of this manuscript is to demonstrate that Fourier analysis, when combined in the right manner with the said quantitative non-divergence estimates, can be as powerful and versatile as dynamical methods, in proving results on Diophantine approximation for smooth manifolds. In fact, quite pleasantly, not only do our methods recover
the main result from \cite{BY} (see Remark \ref{rem Kh rev}), but they also lead to new refinements of Khintchine's theorem for smooth nondegenerate manifolds, pertaining to both the 
Hausdorff dimension and Hausdorff measure of $\tau$- approximable points on them. Moreover, our lower bounds for Hausdorff dimension generalize the one in \cite{Bers12} from analytic to 
sufficiently smooth manifolds
which addresses a question of
Beresnevich \cite[p. 231]{Bers12}. \\
\\
Next, our focus shall be 
on the following problem stated 
in \cite{BY}.

\begin{problem}[Dimension Problem]\label{prob dim conj}
Let $1\le d<n$ be integers and $\widetilde{\cM}_{n,d}$ be the class of submanifolds $\cM\subset\R^n$ of dimension $d$ which are nondegenerate at every
point. Find the maximal value $\tau_{n,d}$ such that
\begin{equation}\label{dimconj}
 \dim\cS_n(\tau)\cap\cM=\frac{n+1}{\tau+1}- \codimension \cM\quad\text{ whenever }
1/n\le\tau<\tau_{n,d}
\end{equation}
for every manifold $\cM\in \widetilde{\cM}_{n,d}$. 
\end{problem}

We refer to \cite{Bers12} and \cite[\S1.6.2]{BRV16} for an explanation about the form of the expression \eqref{dimconj}, which arises from volume-based heuristics
for the number of rational points 
lying close to $\cM$. 
Problem \ref{prob dim conj} was solved 
for nondegenerate planar curves 
in \cite{BDV07} and \cite{BZ}, 
with $\tau_{2,1}=1$. We encourage the 
reader to consult \cite{BDV07, BZ, HuangRPAdv, Vaughan-Velani-2007} 
for further generalizations, refinements 
and extensions. The problem remains 
open for curves in dimensions $n\ge3$ 
and subclasses of nondegenerate manifolds in $\R^n$ of every dimension $d<n$. \par

Consider the well known example (see \cite[Example 1.3]{Bers12} and \cite[\S 2]{BY}) provided by the manifold
\begin{equation*}
\cM:=\{(x_1,\dots,x_d,x_d^2,\dots,x_d^{n+1-d})\in\mathbb{R}^n: (x_1,\dots,x_d)\in \mathscr{U}_d\}.
\end{equation*}
It is easily verified that $\cM$ is nondegenerate, and since it contains a (compact piece) of a $d-1$ dimensional hyperplane, $\dim\cS_n(\tau)\cap\cM\ge \dim\cS_{d-1}(\tau)$. Therefore, by \eqref{eq JaBe}, for $\tau> \frac{1}{n-d}$, we have that
$$
 \dim\cS_n(\tau)\cap\cM\ge \frac{d}{\tau+1}>\frac{n+1}{\tau+1}-\text{ codim }\cM\,.
$$
Consequently, $\tau_{n,d}\le \frac{1}{n-d}$ for $d>1$.  Observe that this example is not applicable to curves in higher demensions (when $d=1$). In \cite{BY}, it was conjectured that within the class $\widetilde{\cM}_{n,d}$ of nondegenerate manifolds of dimension greater than one, this upper bound is exact. For curves, a different lower bound was conjectured, informed by estimates obtained in \cite[Theorem 7.2]{Bers12} and \cite{BVVZ21}.

\begin{conjecture}\label{conj dim nondeg}

(i) Let $1<d<n$. Then $\tau_{n,d}=\frac{1}{n-d}$.

(ii) For $n\geq 2$, $\tau_{n, 1}=\frac{3}{2n-1}$.
\end{conjecture}

\subsubsection{Lower Bounds for Hausdorff dimension}
It was established in \cite[Theorem 1]{BLVV17} that for any $C^2$ submanifold $\cM\subset\R^n$ of dimension $d$ (notice the absence of a nondegeneracy assumption), we have
\begin{equation}\label{eq BLVV17}
 \dim\cS_n(\tau)\cap\cM\ge \frac{n+1}{\tau+1}-\textrm{ codim }\cM\,,
\end{equation}
for $\frac1n\le\tau<\frac{1}{n-d}$.
In the special case of analytic nondegenerate submanifolds of $\R^n$, the following version was obtained in \cite{Bers12}.

\begin{theorem}[Theorem~2.5 in \cite{Bers12}]\label{t2.5}
For every analytic nondegenerate submanifold $\cM$ of $\R^n$ of dimension $d$ and codimension $m=n-d$, any monotonic $\psi$ such that $q\psi(q)^m\to\infty$ as $q\to\infty$ and any $s\in(\frac{md}{m+1},d)$ we have that
\begin{equation}\label{e1.10}
 \cH^s(\cS_n(\psi)\cap\cM)=\infty
\end{equation}
whenever the series
\begin{equation}\label{conv2}
  \sum_{q=1}^\infty \
 q^n\Big(\frac{\psi(q)}q\Big)^{s+m}
\end{equation}
diverges.
\end{theorem}
The same paper establishes an even stronger result for nondegenerate analytic curves (Theorem 7.2), with $m$ replaced by $\frac{2n-1}{3}$ and for $s\in \left(d-\frac{1}{2}, d\right)=\left(\frac{1}{2}, 1\right)$.  This theorem is also behind the second part of Conjecture \ref{conj dim nondeg}. For partial progress towards the remaining upper bound in Conjecture \ref{conj dim nondeg} by counting rational points, we refer the reader to \cite{BVVZ17, Huang19, Huang20, JJHJL, SY22, Sim18}.

\subsubsection{Upper Bounds for Hausdorff dimension}
Next we present our new results pertaning to upper bounds on the Hausdorff measure and dimension of the set of $\psi$-approximable points on manifolds. 

\begin{theorem}
    \label{thm haus s zero}
    Let $n\geq 2$ be an integer, $s>0$ and let $\mathcal{M}$ be a smooth submanifold of $\mathbb{R}^n$ of dimension $d$ and codimension $m$, such that
    \begin{equation}
        \label{eq excep Hausd}
        \mathcal{H}^s\left(\{\bfy\in\mathcal{M}: \mathcal{M} \text{ is not $l$-nondegenerate at }\bfy\}\right)=0.
    \end{equation}
    Let $\alpha:=\frac{1}{d(2l-1)(n+1)}$, and $\psi:(0, \infty)\to (0,1)$ be a monotonic approximation function so that
    \begin{equation}
        \label{eq Hmeas gset conv}
        \sum_{q=1}^{\infty}q^n\left(\frac{\psi(q)}{q}\right)^{s+m}
    \end{equation}
    converges, and there exists an $\epsilon>0$ such that
    \begin{equation}
        \label{eq Hmeas bset conv}
        \sum_{t=1}^{\infty}\left(\frac{\psi(e^t)}{e^{t}}\right)^{\frac{s-d}{2}}(\psi(e^t)^{n-\frac{1}{2}}e^{(\frac{3}{2}-(n+5)\epsilon)t})^{-\alpha}<\infty.
    \end{equation}
    Then
    \begin{equation}
        \label{eq conc Hmeas zero}
        \mathcal{H}^s(\mathcal{S}_n(\psi)\cap \mathcal{M})=0.
    \end{equation}
\end{theorem}
\begin{remark}
\label{rem Kh rev}
Recall that the $d$- dimensional Hausdorff measure $\cH^d$ is a constant multiple of the Lebesgue measure $\mu_d$. Thus, setting $s=d$ in the above theorem establishes the convergence case of Khintchine's theorem for smooth nondegenerate manifolds of arbitrary dimension. This recovers the main result (Theorem 1.2) of \cite{BY}.    
\end{remark}

The following corollary is also a straightforward consequence of Theorem \ref{thm haus s zero} and the lower bound \eqref{eq BLVV17}. It establishes the expected Hausdorff dimension for the set of $\psi_\tau$-approximable points on a smooth nondegenerate manifold, for $\tau$ close to $1/n$, thus making progress towards resolving Problem \ref{prob dim conj}.
\begin{corollary}
\label{cor lbound dim}
Let $n\geq 2$ be an integer, $\tau\in [\frac{1}{n}, 1)$ be a real number, and $\mathcal{M}$ be a smooth submanifold of $\mathbb{R}^n$ of dimension $d$ which is $l-$nondegenerate everywhere except possibly on a set of Hausdorff dimension $\leq \frac{n+1}{\tau+1}-\text{codim }\mathcal{M}$. Suppose that $\tau$ satisfies
\begin{equation}
    \label{eq tau cond}
    \tau<\frac{3\alpha+1}{(2n-1)\alpha+n},
\end{equation}
where $\alpha:=\frac{1}{d(2l-1)(n+1)}$. 
Then
$$\dim (\mathcal{M}\cap\mathcal{S}_n(\tau))=\frac{n+1}{\tau+1}- \codimension \mathcal{M}.$$
\end{corollary}
Note that Beresnevich and Yang \cite[Corollary 2.9]{BY} obtain the same result, but with a range of $\tau$ which is restricted by
$$\frac{n\tau-1}{\tau+1} \leq \frac{\alpha(3-2n\tau)}{2\tau+1}.$$
A short calculation, solving 
the implicit quadratic inequality in $\tau$, 
shows our range \eqref{eq tau cond}
is strictly larger.

\subsubsection{Spectrum of Diophantine exponents}
For $x\in\R$, the exponent of simultaneous rational approximations to $n$ consecutive powers of $x$ is defined as follows
$$
\lambda_n(x):=\sup\left\{\tau>0:(x,x^2,\dots,x^n)\in \cS_n(\tau)\right\}.
$$
This was introduced in \cite{BL05}. 
Dirichlet's theorem implies that $\lambda_n(x)\in[\frac1n,+\infty]$ for any $x\in\R$. The spectrum of $\lambda_n$ is defined to be
$$
\textrm{spec}(\lambda_n):=\lambda_n(\R\setminus\Q)=\left\{\lambda\in\left[\tfrac1n,+\infty\right]:\exists\;x\in\R\setminus\Q\;\text{with}\;\lambda_n(x)=\lambda\right\}.
$$
The following question then arises naturally.
\begin{problem}[Bugeaud-Laurent {\cite[Problem 5.5]{BL07}}]
\label{prob dio exp}
Is\; $\mathrm{spec}(\lambda_n)=[\frac1n,+\infty]$\;?
\end{problem}
In \cite{BB20}, it was established that
$$
\left[\tfrac{n+4}{3n},+\infty\right]\subset\textrm{spec}(\lambda_n)\quad\text{for every $n\ge3$},
$$
leading to a partial resolution of Problem \ref{prob dio exp}. We refer the reader to \cite[\S 2.3]{BY} for a Hausdorff dimensional version of the above problem, and an exhaustive list of references for historic progress in the area. 

Corollary 2.10 in \cite{BY} was the first instance of closing the gap in the spectrum of $\lambda_n$ from the other end, close to the Dirichlet exponent $1/n$. 
Quite pleasantly, Corollary \ref{cor dio exp} not only lets us improve their exponent, but also provide a clean and explicit range for the spectrum.

\begin{corollary}
\label{cor dio exp}
Let $n\geq 2$ be an integer and let $\tau\geq \frac{1}{n}$ satisfy
\begin{equation}
    \label{eq spec tau cond}
    \tau<\frac{2n^2+n+2}{(n^2+n+1)(2n-1)}.
\end{equation}
Let $\mathcal{C}$ be a smooth curve in $\mathbb{R}^n$ which is nondegenerate everywhere except possibly on a set of Hausdorff dimension $\leq \frac{n+1}{\tau+1}-n+1$. Then
$$\dim (\mathcal{C}\cap\mathcal{S}_n(\tau))=\frac{n+1}{\tau+1}-n+1.$$
In particular, for every $n\geq 3$, we have
\begin{equation}
    \label{eq spec}
    \left[\frac{1}{n}, \frac{1}{n}+\frac{n+1}{n(2n-1)(n^2+n+1)}\right)\subset \mathrm{ spec }(\lambda_n).
\end{equation}
\end{corollary}
For comparison, \cite[Corollary 2.16]{BY} states that for every $n\geq 3$, $\left[\frac{1}{n}, \frac{1}{n}+\frac{\delta_n}{n}\right]\subset \mathrm{ spec }(\lambda_n)$, where $\delta_n\in \left(\frac{1}{2n^2+6n}, \frac{1}{2n^2+5n}\right)$, which is implicitly defined in their equation (2.21). Since 
$$\frac{1}{2n^2+5n}<\frac{n+1}{(2n-1)(n^2+n+1)},$$ the range in \eqref{eq spec} is strictly larger.

\section{Outline of the Proof and Novelties}
Fourier analytic techniques, 
e.g. the circle method,
have long been used 
to count rational points 
on or near manifolds.
Further, quantitative non-divergence 
estimates in the space of lattices are powerful tools from
homogeneous dynamics which
have found versatile application in the area.
These methods are classical and have played an important role 
in Diophantine analysis
for decades now.
The works of 
Kleinbock and Margulis 
\cite{KM 1998, KM 1999}
from the late 90's
are fine testimonies to their strength.

To the best of our knowledge,
Fourier analysis 
and homogeneous dynamics have 
been used in a mutually 
exclusive manner 
when counting rational points.
The novelty of this work lies in 
a successful marriage between the two types 
of techniques. Further,
we do not (explicitly) use the Dani correspondence
and the geometry of numbers,
which is central to the work of 
Beresnevich and Yang \cite{BY}
and their use of quantitative
non-divergence estimates. 
In contrast, 
the foundation of our approach 
is a re-interpretation 
of quantitative 
non-divergence in terms of the 
Fourier transform of the surface measure associated with the submanifold under consideration. 
To explain this properly, we introduce more notation.

Let $\nu$ be the surface measure on 
$\cM$ (the push-forward of the Lebesgue measure under $\bfF$). We shall also weigh 
the surface measure 
by $\Omega$, but 
suppress it from the notation for ease of exposition. 
The Fourier transform 
of $\nu$ is defined as
\begin{equation*}
\widehat{\nu}(\boldsymbol{\xi} ):=
\int_{\bR^n} e(-\langle
\bfz, \boldsymbol{\xi}  \rangle )\, \rd\nu(\bfz) 
=\int_{\bR^d} e(-\langle (\bfx, \bfF(\bfx)), \boldsymbol{\xi}\rangle )\, \Omega(\bfx)\,\rd\bfx.
\end{equation*}
First we demonstrate, 
by Poisson summation 
and partial integration, that 
\begin{equation}\label{eq Fourier transform}
N_{\bfF}^{\bft}(\delta,Q) - c_\bft \delta^m Q^{d+1} 
\approx 
\delta^{m} Q^{d+1}
\sum_{\substack{(\bfv,c) \in \bZ^{n+1}\\
\Vert \bfv\Vert_{\infty},\vert c\vert  \sim \delta^{-1}}} 
\int_\bR \widehat{\nu}(Qy\bfv) e(-Qcy) y^d \omega(y) \rd y.
\end{equation}
To show that
the right hand side above
is an error term 
(in the relevant ranges of $\delta$), 
we proceed as follows.
For each $\bfv$, the size of 
the Fourier transform
$\widehat{\nu}(\boldsymbol{\bfv})$
is determined by a small set of 
points on $\cM$. 
In fact, the underlying 
oscillatory integral, given by
$$\widehat{\nu}(Qy\bfv)=
\int_{\bR^d} e(-\langle (\bfx, \bfF(\bfx)), Qy\bfv \rangle )\, \Omega(\bfx)\rd\bfx,$$
decays rapidly except
for those values of $\bfx$ for which 
\begin{equation}\label{eq nabla}
    \vert \nabla_{\bfx}
\langle (\bfx, \bfF(\bfx)),
\bfv \rangle \vert 
\end{equation} 
is small. For our purposes it 
would be ideal if those
parts of the manifold were
just removed.
Our argument does
precisely that.
For us, quantitative non-divergence 
in the space of lattices
is a device that estimates the measure
of those critical regions on $\cM$ 
that prevent $\widehat{\nu}(\bfv Qy)$
from decaying rapidly,
whenever the integer vector 
$\bfv\neq \bzero$ is of a given norm. To make use of this, we also capitalize
on the oscillation in the $y$ variable.
Observe that the integral
in $y$ decays rapidly unless 
\begin{equation}\label{eq dist int}
    \vert \langle (\bfx, \bfF(\bfx)),
\bfv \rangle - c \vert 
\end{equation}
is quite small.

It turns out  
Bernik, Kleinbock, and Margulis
have provided a quantitative 
non-divergence estimate (see Theorem \ref{thm quant non div}) that 
controls the measure of
the ``critical'' points $\bfx$
for which \eqref{eq nabla}
and \eqref{eq dist int}
are simultaneously small. 
The contribution 
of these regions to counting 
rational points will be handled 
by measure theoretic 
and mild structural 
information: 
we can cover this sublevel set
by a number of small balls.
Actually, we use smooth functions 
whose support are balls. 
To construct a
suitable partition of unity, we
use arguments inspired by 
Beresnevich and Yang \cite{BY},
based on Taylor expansions 
and Besicovitch's covering theorem.

Outside of the critical sublevel set, 
\eqref{eq nabla}
or \eqref{eq dist int}
is somewhat large, as dictated by the choice of certain parameters. 
To exploit this,
we use an integration by parts argument, which 
in our situation is
non-standard in two aspects. 
Firstly, because we partition
the relevant unity function into smooth pieces
with small support, 
we encounter amplitude functions 
with large (``rough'') derivatives. 
This makes the integration
by parts rather delicate and 
the correct choice of the parameters 
is essential here. 
The second issue
is that we don't have precise information about ``which'' gradient (whether in $\bfx$ or $y$) is large. In other words, the variables
$\bfx$ and $y$ play distinguished roles and we need a careful decomposition to handle both cases separately.

\subsubsection*{Structure of the paper} In Section \ref{preliminaries} we state useful lemmata as well as a version of quantitative non-divergence, which will be used later in the proof of our main theorems. Section \ref{sublevelFA} is the major technical part of this article. We first perform the sub-level set decomposition (in \S\ref{subsec sublevel}), and then express the counting function $N_\bfF^{\bft}(\delta,Q)$ Fourier-analytically as a sum of oscillatory integrals (in \S \ref{subsec fourier expansion}). The heart of the section is in \S\ref{subsec ibp}, in which we estimate the contribution of these integrals using a delicate integration by parts argument. With these estimates in hand, we complete the proof of Theorems \ref{thm main asymp bounds}-\ref{thm main upper bounds} in Section \ref{parameters}. Finally, in Section \ref{proofsapplications} we give details for the proofs of our applications to Diophantine approximation.\par

Since $\bfF$ and $\bft$
are fixed throughout 
our arguments, we shall
not emphasize 
the dependence of 
$N_\bfF^{\bft}(\delta,Q)$
on these parameters --- 
aside from stating the theorems. 
Thus, for most 
of the manuscript we
shall simply write $N(\delta,Q)$
instead of $N_\bfF^{\bft}(\delta,Q)$.\par

\section{Preliminaries}\label{preliminaries}
We use the following 
truncated Fourier expansion.
\begin{lemma}\label{lem: truncated Poisson}
Let $w: \bR \rightarrow \bR$ 
be smooth and $\supp(w)\subseteq(-1,1)$. Assume that $w(x)=w(-x)$ for all $x\in \mathbb{R}$. 
If $\eta,A>0$,
then 
$$
w\Big(\frac{\Vert x\Vert}
{\delta}\Big)
= 
\delta \widehat{w}(0)
+ \delta  
\sum_{1 \leq \vert j \vert 
\leq Q^{\eta}/\delta} 
\widehat{w}(\delta j) e(jx)
+ O_{w,\eta,A}(Q^{-A})
$$
uniformly in $\delta \in (0,1/2)$ 
and $x\in \bR$.
\end{lemma}
\begin{proof}
This is a consequence of consecutive applications of
Poisson summation,
a change of variables,
dyadic decomposition,
and partial integration. 
See \cite[Lemma A.3]{ST} for the details.
\end{proof}
Further, we need the following (weak) variant of 
Besicovitch's Covering Theorem (see for example \cite[Theorem 2.7]{Mat}) ensuring that any 
bounded set can be covered economically 
by balls of a given radius, in the sense that
the balls do not overlap by more than a dimensional constant.
 \begin{theorem}[]\label{thm: Besicovitch}
 Let $r\in (0,1)$, and
 $n\geq 1$ be an integer.
 Then there exist constants $D_n$ and $D_n'$, which only depend on $n$ with the following property. For
 any bounded subset $\sA\subset \bR^n$,
 there exists a family $\mathrm{Bes}(\sA,r)$
 of balls $\sB(\bfx,r)$, with $\bfx \in \sA$,
 so that 
 $$
 \mathds{1}_\sA (\bfy) \leq 
 \sum_{\sB(\bfx,r) \in 
 \mathrm{Bes}(\sA,r) } 
 \mathds{1}_{\sB (\bfx,r)}(\bfy)
 \leq D_n
 \qquad \mathrm{for\, all\,}\bfy \in \mathbb{R}^n.
 $$
 Furthermore, we have
  $$
\sum_{\sB(\bfx,r) \in 
 \mathrm{Bes}(\sA,r) } 
 \mathds{1}_{\sB (\bfx,2r)}(\bfy)
 \leq D_n'
 \qquad \mathrm{for\, all\,}\bfy \in \mathbb{R}^n.
 $$
 
\end{theorem}
\begin{proof}
See Mattila's monograph 
\cite[Theorem 2.7]{Mat}.
\end{proof}

Let $h\in C_c^\infty (\R^n)$ and $\psi\in C^\infty(\R^n)$ be real valued such that
$\nabla\psi\neq 0$ on the support of $h$. 
We define the differential operator $L$
via
\begin{equation}
 \label{def Lop} 
 L h
= 
\text{div} 
\Big( \frac{h \nabla \psi }{|\nabla\psi|^2}
\Big).
\end{equation}
Then the formal adjoint operator is given by 
$ L^*=
- \langle \frac {\nabla\psi}{|\nabla\psi|^2}, 
\nabla \rangle $. Importantly, it has the property
that
$(2\pi \mathrm{i} \lambda) ^{-1}L^* e(\lambda \psi)
=e(\lambda \psi)$.

We let $L^0 h:=h$ and put
$L^{\circ N} h:=L \circ (L^{\circ(N-1)} h)$
for any integer $N\geq 1$.
We then have by integration by parts 
$$
\int_{\R^n} e(\lambda \psi(\mathbf{z})) h(\mathbf{z}) d\mathbf{z} = 
(2\pi \mathrm{i}\lambda )^{-N} 
\int_{\R^n} e(\lambda \psi(\mathbf{z})) (L^{\circ N} h)(\mathbf{z}) d\mathbf{z} 
$$ for $N\geq 1$.
 
The goal is now to estimate the right hand side efficiently, 
for which we need control on the function $L^{\circ N} h$. 
First we introduce terminology inspired by
Anderson, Cladek, Pramanik, and Seeger \cite[Appendix A]{ACPS}.

\begin{definition} \label{def typeABterms}
Let $j\geq 0$ be an integer.\\ 
(i) A smooth function $g$ is said to be 
of {\emph{amplitude type and order j}}
if 
$g = h_j/|\nabla\psi|^j$ 
where $h_j$ is a (multi-variable) derivative 
of order $j$ of $h$. \\
(ii) 
A smooth function $g$ is said to be 
of {\emph{phase type and order j}} 
if it is equal to $1$, for $j=0$,
or $g= \psi_{j+1}/|\nabla\psi|^{j+1}$ 
for $j\geq 1$ 
where $\psi_{j+1}$ is a (multi-variable) 
derivative of order $j+1$ of $\psi$. 
\end{definition}

\begin{lemma}[Partial Integration]\label{lem partial}
Let $n,N\geq 0$ be integers, $h\in C_c^N (\R^n)$ and $\psi\in C^{N+1}(\R^n)$. Suppose that $\nabla \psi\neq 0$ on the support of $h$. Then 
there exists an integer $V=V(N,n)\geq 1$ so that
$$
L^{\circ N} h = \sum_{\nu=1}^{V} c_{N,\nu} h_{N,\nu}.$$ 
Here $c_{N,\nu}$ are complex coefficients, and each $h_{N,\nu}$ is of the form
$$
P(\tfrac{\nabla \psi}{|\nabla\psi|}) 
\beta_{\mathrm{amp}}
\prod_{l=1}^M \gamma_l,$$
where $P$ is a polynomial of $n$ variables 
(independent of $h$ and $\psi$), $\beta_{\mathrm{amp}}$ 
is of amplitude type with order $j_{\mathrm{amp}}$
for some $j_{\mathrm{amp}}\in \{0,\dots, N\}$, while
each $\gamma_l$ 
is of phase type and order $\kappa_l\geq 1$. We also have
\begin{equation}
    \label{eq order}
    j_{\mathrm{amp}}+\sum_{l=1}^M \kappa_l=N.
\end{equation} The terms $P, \beta_{\mathrm{amp}}, \gamma_l$ 
depend on $\nu$.
\end{lemma}
\begin{proof}
This follows from multi-dimensional partial integration, with a careful book keeping,
and induction on $N$; cf.
Anderson, Cladek, Pramanik, and Seeger 
\cite[Lemma A.2]{ACPS}.
\end{proof}
We make the following observation for subsequent use. Since $\kappa_l>0$ for each $l\leq M$, \eqref{eq order} implies that
\begin{equation}
    \label{eq jM sum}
    j_{\mathrm{amp}}+M\leq N.
\end{equation}

The following simple lemma will also be useful later.

\begin{lemma}\label{le rationals in a ball}
Let $q\in \N$ and $r >0$ be such that $qr \geq 1$. 
Then 
$$
\# (q^{-1}\Z^d \cap\sB(\bfc,r)) \leq (3 r q)^{d}
$$
for any choice of $\bfc \in \R^d$.
\end{lemma}
\begin{proof}
The quantity in question is at most 
$$ 
\# \Big( \Z^d \cap \prod_{i\leq d}[qc_i-qr,qc_i+qr]\Big)
\leq (2qr+1)^d.
$$
By using $qr \geq  1$, 
the claim follows readily.
\end{proof}

To proceed, we use a 
theorem of Bernik, Kleinbock, 
and Margulis \cite[Theorem 1.4]{BKM}.
To this end, we need to introduce 
more notation.
Given a vector $\bfv \in \R^n$,
the function 
\begin{equation}\label{def Phi}    
\Phi_{\bfv}(\bfx):=
\langle (\bfx, \bfF(\bfx)), \bfv\rangle
\end{equation}
shall play a central role. 
For $0<\Delta \leq 1$, $K>0$ and $T\geq 1$ 
we define the sub-level set
\begin{equation}\label{def non osc set}
    \mathfrak{S}_{\Delta, K, T}:= \left\{
    \bfx\in \sU_d: 
    \,\mathrm{there\,exists}\, 
    \bfv\in \Z^{n} 
     \,\mathrm{so\,that}\,
    \begin{array}{l} \\
   \Vert\Phi_{\bfv}(\bfx) \Vert 
   < \Delta, \\
   \|\nabla \Phi_{\bfv}(\bfx)
   \|_\infty < K, \\
   0<\|\bfv\|_\infty < T
  \end{array}
    \right\}.
\end{equation}

\begin{theorem}[Quantitative Non-Divergence]
\label{thm quant non div}
Let $\bfx_0\in \sU_d$ and suppose 
$\bfF: \sU_d\to\R^n$ is $l$--nondegenerate 
at $\bfx_0$. For $\Delta\in (0,1]$, $K>0$, 
and $T\geq 1$, let $\fS_{\Delta, K,T}$
be as in \eqref{def non osc set}. If  
\begin{equation}
    \label{C1}
    \tag{C1}
    \Delta^n < K T^{n-1},
\end{equation}
then there exists a ball $\sB_0\subset\sU_d$ 
centred at $\bfx_0$ and a constant $E\ge1$ such that 
\begin{equation*}
\label{eq nondiv loc}
\mu_d(\fS_{\Delta, K,T}\cap \sB_0) \leq E \cdot  
( \Delta K T^{n-1})^{\frac{1}
{d(2l-1)(n+1)}}\mu_d(\sB_0).   
\end{equation*}

In particular, if $\bfF$ is $l$-nondegenerate in an open neighborhood of $\overline{\sU_d}$ and condition \eqref{C1} 
is satisfied, then it follows via a compactness argument that there exists some constant $C>0$ such that 
\begin{equation*}
\label{eq nondiv glob}
\mu_d(\fS_{\Delta, K,T}) \leq C \cdot  
( \Delta K T^{n-1})^{\frac{1}{d(2l-1)(n+1)}}.
\end{equation*}
\end{theorem}
\begin{proof}
See \cite[Theorem 1.4]{BKM}.
\end{proof}
Note that the measure bounds provided by
Theorem \ref{thm quant non div} 
are non-trivial if
\begin{equation}
    \label{C1b}
    \Delta K T^{n-1} \leq 1.
\end{equation}
In our application we will set 
\begin{equation}
    \label{eq Tdef}
    T:=2C_\bfF Q^{\eta} /\delta,
\end{equation}
where $C_{\bfF}$ a constant only depending on $\bfF$ (to be introduced in Lemma \ref{le error small}), and $\eta>0$ is a fixed parameter. 
We shall keep track of the exact dependence on $\eta$. 
For a parameter-vector
\begin{equation}\label{def parameter vec}
    \bfp:=(\Delta,K)\in 
    (0,1/2]\times (0,1],
\end{equation}
we define 

\begin{equation}
\label{eq defSp}
\mathfrak{S}_\bfp := \mathfrak{S}_{\bfp,T},
\end{equation}
with $T=2 C_{\bfF} Q^{\eta}/\delta$.

\section{Sub-level Set Decomposition and Fourier Analysis}
\label{sublevelFA}
\subsection{Sub-level Set Decomposition}
\label{subsec sublevel}
The objective of this subsection is to partition the set $\mathscr{U}_d$, via a smooth partition of unity, into two regions. We aim to achieve this by decomposing a weight supported on the unit ball into a sum of two smooth functions: one which essentially lives on the sub-level set $\mathfrak{S}_\bfp$ whose measure can be estimated by Theorem \ref{thm quant non div}, and the second supported on a complementary set where Fourier analytic techniques can be employed. To this effect, we construct a smooth cut-off function associated with $\mathfrak{S}_\bfp$.

\begin{lemma}[Smooth cut-off]\label{le smooth cut off}
Let $\fS_\bfp$ be as in (\ref{eq defSp}).
If \begin{equation}
    \tag{C2}
    \label{C2}
    0<r\leq  Q^{-2\eta}\min(K\delta,\Delta/K),
\end{equation}
then there exists a smooth function 
$W_{r}:=W_{r,\fS_\bfp}:\sU_d \rightarrow [0,1]$
with the following properties. 
\begin{enumerate}
    \item On $\fS_\bfp$, the function $W_r$ is identically one.
    \item The function 
    $W_r$ is $r$-rough, that is
         $\Vert W_r^{(\balp)}\Vert_{\infty} 
         \ll_\balp r^{-\Vert\balp\Vert_1 }$ 
         for any multi-index $\balp \in \Z_{\geq 0}^d$.
    \item If $Q$ is large enough, then the support 
         of $W_r$ is contained in $ \fS_{2\bfp}$. Further, $\supp (W_r)$
         is contained in a union of $O(t_r)$ many closed balls of radius $r$ 
         where $t_r:=\mu_d(\fS_{2\bfp})r^{-d}$.
\end{enumerate}
\end{lemma}
\begin{proof}
Fix a smooth function 
$b: \R^d \rightarrow [0,1]$
with 
$\supp(b) \subset \sU_d$ and such that $b$ is equal to one on $\sB(\mathbf{0},1/2)$. Let
$
\fB:=\mathrm{Bes}(\mathscr{U}_d,r/2)
$ 
be a Besicovitch cover of $\sU_d$ obtained by applying
Theorem \ref{thm: Besicovitch}, and set
$
\fC:=\{ \sB(\bfc,r/2)\in \fB: \sB(\bfc,r)\cap \fS_\bfp \neq  \emptyset \}$.
To localise to a neighborhood of $\fS_{\bfp}$ 
(and then normalise thereafter),
we consider for $\bfx\in \sU_d$ the functions
$$
W_{\fS_{\bfp}}(\bfx):=
\sum_{\substack{\sB(\bfc,r/2)\in \fC}} 
b\Big( \frac{\bfx - \bfc}{r}\Big)
\,\, \mathrm{and} \,\,
W_{\mathrm{norm}}(\bfx):=
\sum_{\sB(\bfc,r/2)\in \fB} 
b\Big( \frac{\bfx - \bfc}{r}\Big).
$$

Notice that $W_{\mathrm{norm}} \geq 1$
on any ball in the set $\fB$.
Next, we ensure that $W_{\fS_{\bfp}}$ 
is actually equal to one
on $\fS_{\bfp}$.
This is an issue because the cover $\fC$ of $\fB$
can have balls overlapping, with large multiplicity.
To remedy this, we define on $\sU_d$ the function
$W_r:=W_{\fS_{\bfp}}/W_{\mathrm{norm}}$.
Clearly, $W_r$ is a smooth function on $\sU_d$.
Since $ \fC \subseteq \fB$ and $b$ 
is non-negative, we have $0\leq W_r \leq 1$.
Now if $\bfx\in \fS_\bfp$ then 
$$
W_{\mathrm{norm}}(\bfx)=
\sum_{\substack{ \sB(\bfc,r/2)\in \fB \\
\Vert \bfx -\bfc \Vert_2 < r }} 
b\Big( \frac{\bfx - \bfc}{r}\Big)
= 
\sum_{\substack{\sB(\bfc,r/2)\in \fC}} 
b\Big( \frac{\bfx - \bfc}{r}\Big)
= W_{\fS_{\bfp}}(\bfx).
$$
Hence, $W_r = 1$ on $\fS_{\bfp}$.

Let us now establish the second property.
By the construction of $\fB$ and the second part of Theorem \ref{thm: Besicovitch}, we know that
at any given point $\bfx$ the function
$W_{\mathrm{norm}}(\bfx)$ is a sum of at least one but 
at most $O(1)$
many $r$-rough functions.
Here the implicit constant in $O(1)$ depends
only on the dimension $d$. As a result, $W_{\mathrm{norm}}$
is $r$-rough. Similarly, we see that
$W_{\fS_{\bfp}}$ is $r$-rough. An application of the 
Leibnitz formula implies that $W_r$ is therefore $r$-rough, using the lower bound for $W_{\mathrm{norm}}$ on $\sU_d$.

Next, we verify the third property, where we start with proving that $\supp (W_r) \subset \fS_{2 \bfp}$.
Notice that any 
$\bfx\in \supp (W_r)$ can be written as 
$\bfx = \bfx_0 + \bfh$
with $\bfx_0\in \fS_{\bfp}$ and $\Vert \bfh \Vert_2 \leq 2r $.
By using a Taylor expansion for $F$ component-wise, we infer 
$$
F_i(\bfx) = F_i(\bfx_0) +
\langle \nabla F_i(\bfx_0),\bfh\rangle
+ O(r^2)
$$
where the implied constant depends only on $F_i$. For $\bfv \in \R^n$, recall the definition of $\Phi_{\bfv}$ from \eqref{def Phi}.
By multiplying with 
$v_i\ll Q^\eta/\delta$, summing 
over all $i\leq m$,
and by using $\bfx_0\in \fS_{\bfp}$, we conclude 
\begin{align*}
\Phi_{\bfv} (\bfx) 
& = \sum_{i\leq d} v_i x_i 
+ \sum_{1\leq i \leq m} v_{d+i} F_i(\bfx)\\
& = \sum_{i\leq d} v_i (x_{0,i} +h_i)
+ \sum_{1\leq i \leq m} v_{d+i} 
(F_i(\bfx_0) +
\langle \nabla F_i(\bfx_0),\bfh\rangle
+ O(r^2))\\
&=\Phi_{\bfv} (\bfx_0)+ \langle \nabla\Phi_{\bfv} (\bfx_0), \bfh \rangle+O(r^2Q^\eta/\delta)\\
& = \Phi_{\bfv} (\bfx_0) + 
O(K r+ r^2 Q^\eta/\delta).
\end{align*}
Here, the last equality follows from the fact that $\|\nabla \Phi_{\bfv}(\bfx_0)\|_\infty < K$, since $\bfx_0\in \fS_{\bfp}$. For the same reason, we also know that $\Vert\Phi_{\bfv}(\bfx_0) \Vert 
   < \Delta$. Now, the assumption on $r$ ensures 
$r^2 \del^{-1} \leq  rK \leq \Delta Q^{-2\eta}$. 
The upshot is that for any 
$\bfx\in \supp (W_r)$, we have
\begin{equation}\label{eq Phi_a stable}
\Vert \Phi_{\bfv} (\bfx) \Vert 
<  2\Delta
\end{equation}
as soon as $Q$ is large enough.
Similarly, for $\bfx = \bfx_0 + \bfh \in \supp (W_r)$,
Taylor expansion shows
$$
\Vert \nabla F_i(\bfx) -  \nabla F_i(\bfx_0) 
\Vert_\infty  \ll r
$$
So,
$ \Vert\nabla \Phi_{\bfv} (\bfx) \Vert_\infty 
\leq \Vert \nabla \Phi_{\bfv} (\bfx_0) \Vert_\infty + 
O(r Q^\eta \delta^{-1})
$.
As $r Q^\eta \delta^{-1}
\leq K Q^{-\eta}$,
we infer that
\begin{equation}\label{eq K stable}
\Vert \nabla \Phi_{\bfv} (\bfx) \Vert_\infty 
\leq 2 K
\end{equation}
for any $\bfx\in\fS_\bfp$ whenever $Q$ is large.
Combining \eqref{eq Phi_a stable}
and \eqref{eq K stable}, 
implies $\supp (W_r) \subset \fS_{2 \bfp}$.\par
For the last part of the third property note that we have
$$r^d\# \fC  \ll \int_{\R^d}\sum_{\substack{\sB(\bfc,r/2)\in \fC}} 
b\Big( \frac{\bfx - \bfc}{r}\Big) 1_{\cU_d} (\bfx) \d \bfx \ll  \mu_d(\fS_{2\bfp}) .$$
Here we used that every ball $\sB(\bfc,r)$ with $\sB(\bfc,r/2)\in \fC$ has a center that is contained in $\sU_d$ and hence the volume of the intersection of such a ball with $\sU_d$ is $\gg_d r^d$.
We infer that
$\# \fC \ll \mu_d(\fS_{2\bfp})/\mu_d(\sB(\bzero,r))
\ll t_r$. 
\end{proof}

From now on we assume that we are given an admissible triple $(\Omega, \omega, w)$ of weights. The following definition is central for us.
\begin{definition}
The $(\bfF,\bfp)$-sub level 
part of $\Omega$ is defined to be
$\wbad :=\Omega(\bfx) W_r(\bfx)$.
The $(\bfF,\bfp)$-good part of 
$\Omega$ 
is defined as 
$\wgood :=\Omega (\bfx) - \wbad (\bfx)$.
\end{definition}
In the following, we shall suppress notation and drop 
the weight functions $\omega$,
and $w$ as well as $\bfF$,

writing
$N^{\Omega}$ in place of 
$N^{\Omega,\omega,w}_\bfF$.
Observe that $\Omega= \wgood + \wbad$.
Since $\wbad$ is non-negative,
our starting point for proving 
Theorem \ref{thm main lower bounds}
is simply that 
\begin{equation}\label{eq basic lower bound}
    N^{\Omega}(\delta,Q)
    \geq    
    N^{\wgood}(\delta,Q).
\end{equation}
\subsection{Fourier Expansion}
\label{subsec fourier expansion}
Our next step is to Fourier expand the right hand side of 
\eqref{eq basic lower bound}.
To this end, let 
\begin{equation}\label{def J}
    J:= \frac{Q^{\eta}}{\delta}.
\end{equation}
We also define the associated index set of frequency 
vectors 
\begin{equation}\label{def j}
    \cJ:= 
    \{\bfj=(j_1,\ldots,j_m)\in \Z^m:
    0\leq \vert j_i\vert \leq J
    \mathrm{\, for\, all\,} 1\leq i \leq m,\,\,
    \bfj \neq \mathbf{0}
    \}.
    \end{equation}
We can now make the following decomposition.
\begin{lemma}\label{le decomp}
Let $A>1$.
Uniformly for 
$\delta\in (0,1/2)$ 
and $Q\geq 1$, 
we have 
\begin{equation} \label{eq Fourier decomp}
N^{\wgood}(\delta,Q)
=
M(\delta, Q)
+ E(\delta,Q)
+ O_{w,\eta,A}(Q^{-A})
\end{equation}
where 
\begin{equation}\label{def main term}
 M(\delta, Q):=   
 \delta^m (\widehat{\bump}(0))^m
\sum_{(q,\bfa)\in \Z^{d+1}} 
\wgood\left(\frac{\bfa}{q}\right) 
\ome\left(\frac{q}{Q}\right),
\end{equation}
and 
\begin{equation}\label{def Error Term}
    E(\delta,Q):=
    \delta^m   
    \sum_{(q,\bfa,\bfj)\in \Z^{d+1}\times \cJ} 
    \wgood\left(\frac{\bfa}{q}\right) 
    \ome\left(\frac{q}{Q}\right)
    B(\delta \bfj)
    e(q \langle 
    \bfF( \bfa/q ),
    \bfj \rangle),
\end{equation}  
with $\cJ$ is defined as in \eqref{def j}.
\end{lemma}
\begin{proof}
By Lemma \ref{lem: truncated Poisson},
we have 
\begin{align*}
\prod_{i\leq m} 
\bump \bigg(
\frac{\|q F_i (\bfa/q) \|}{\delta}
\bigg)
& =
\prod_{i\leq m} \Big(
\delta \widehat{\bump}(0)
+ \delta  
\sum_{1 \leq \vert j_i \vert 
\leq J} 
\widehat{\bump}(\delta j_i) 
e(j_iq F_i (\bfa/q ))
+ O_{\bump,\eta,A}(Q^{-A})
\Big) \\
& = \delta^m (\widehat{\bump}(0))^m
+ \delta^m \sum_{\bfj \in \cJ} 
B(\delta \bfj)
e(q \langle 
\bfF( \bfa/ q ),
\bfj \rangle)
+ O_{\bump,\eta,A}(Q^{-A}),
\end{align*}
where the coefficients are given by
$$
B(\delta \bfj):= \prod_{i\leq m}
\widehat{\bump}(\delta j_i).
$$
Substituting the above into the definition of $N^{\wgood}$ and
recalling \eqref{def counting func} now yields the desired decomposition.
\end{proof}

Now we shall evaluate $M(\delta, Q)$, which shall
naturally turn out to be the main term.
\begin{lemma}\label{le zero mod comp}
Assume that $r \geq Q^{\eta -1}$. If $A>0$, then 
$$
\sum_{(q,\bfa)\in \Z^{d+1}} 
\wgood\left(\frac{\bfa}{q}\right) 
\ome\left(\frac{q}{Q}\right)
= c_{\Omega,\omega} Q^{d+1}
+O_{A,\eta}(Q^{-A})
$$ 
where $$
c_{\Omega,\omega}
:= 
\Big( 1+ O(\mu_d(\fS_{2\bfp})) 
\Big)
\int_{\R^d} \Omega(\bfx) \rd \bfx
\cdot 
\int_{\R} x^d\omega(x) \rd x.
$$
\end{lemma}
\begin{proof}
For brevity, let $G$ denote the Fourier transform
of $\wgood$.
Fix $q \asymp Q$. By Poisson summation
and a change of variables,
\begin{equation}\label{eq zero mode interm}
\sum_{\bfa\in \Z^{d}} 
\wgood\left(\frac{\bfa}{q}\right) 
= 
q^d
\sum_{\bfv \in \Z^{d}} 
G(q \bfv).
\end{equation}
Due to the $r$-roughness of
$\wgood$, its Fourier transform $G$ 
decays rapidly only once the norm of its 
argument exceeds roughly $1/r$. 
Indeed, partial integration implies that
for each $A>1$, we have
$$
G(\bfx)
\ll_A \frac{r^{-A}}{1+\Vert \bfx \Vert_2^A}.
$$
By the assumption on $r$, 
we have $qr \gg Q^{\eta}$
whenever $q\gg Q$.
Thus, 
the right hand side of 
\eqref{eq zero mode interm}
is equal to
$
q^d G(\bzero) + 
O_{A,\eta}(Q^{-A})$.
As a result, 
\begin{align*}
\sum_{(q,\bfa)\in \Z^{d+1}} 
\wgood\left(\frac{\bfa}{q}\right) 
\ome\left(\frac{q}{Q}\right)&= 
G(\bzero)
\sum_{q\in \Z} 
q^d \ome\left(\frac{q}{Q}\right)
+O_{A,\eta}(Q^{-A})\\ 
&= 
Q^d G(\bzero)
\sum_{q\in \Z} 
Z\left(\frac{q}{Q}\right)
+
O_{A,\eta}(Q^{-A})
\end{align*}
where 
$Z(x):=x^d\omega(x)$. 
Arguing as before (with $Z$
in place of $\Omega$), we infer
$$
\sum_{q\in \Z} 
Z\left(\frac{q}{Q}\right)
= Q \widehat{Z}(0) + O_A(Q^{-A}).
$$
The upshot is that
$$
\sum_{(q,\bfa)\in \Z^{d+1}} 
\wgood\left(\frac{\bfa}{q}\right) 
\ome\left(\frac{q}{Q}\right)
=
Q^{d+1} G(\bzero)
\widehat{Z}(0)
+
O_{A,\eta}(Q^{-A}).
$$
Since $\wbad$
is at most one, by construction,
and supported in $\fS_{2\bfp}$,
by Lemma \ref{le smooth cut off},
we observe 
$$
 G(\bzero)
 = 
 \int_{\R^d} \big(\Omega (\bfx) - \wbad (\bfx)\big)\rd \bfx
 =  \int_{\R^d} \Omega (\bfx) \rd \bfx 
 + O(\mu_d(\fS_{2\bfp})).
$$
Using this information in the computation of
$G(\mathbf{0})$ completes the proof.
\end{proof}
For later reference, 
we point out that 
Lemma \ref{le zero mod comp} 
implies 
\begin{equation}\label{eq main term computed fully}
    M(\delta,Q) = 
    c_{(\Omega,\omega,\bump)} 
    [1+ O(\mu_d(\fS_{2\bfp}))] \delta^m Q^{d+1}
    +   O_{A,\eta}(Q^{-A})
\end{equation}
where $c_{(\Omega,\omega,\bump)}$
is the constant from 
\eqref{def leading constant}.

\subsection{Integration by Parts} 
\label{subsec ibp}
It remains to show that
$E(\delta,Q)$
is indeed an error term. Let 
\begin{equation}
    \label{eq v split}
    \bfv = (\bfj,\bfv')\in \mathbb{R}^{m}\times\mathbb{R}^{d}.
\end{equation}
For a large enough constant $C_\bfF>0$ (depending only on $\bfF$ and to be decided shortly), we set
\begin{equation}
\label{def Vset}
\cV:=\{(\bfj,\bfv')\in \Z^n: 
0<\Vert\bfj \Vert_\infty \leq J,\,
0 \leq \Vert\bfv' \Vert_\infty \leq C_\bfF J\}.   
\end{equation}
We define the truncated error term 
$E_{\mathrm{tr}}(\delta,Q)$ to be
\begin{equation}
    \label{def trunc error}
    E_{\mathrm{tr}}(\delta,Q)= 
\delta^m   Q^{d+1}
    \sum_{0 \leq \vert c\vert \leq C_\bfF J}\sum_{\bfv\in \cV} \,
    B(\delta \bfj)
    \int_{\R^{d+1}} y^d
    e(y Q 
    [\Phi_{\bfv}(\bfx) - c] )
    \wgood(\bfx) 
    \ome(y)
    \rd \bfx \rd y.
\end{equation}
The next lemma shows that up to an acceptable error, it is  $E_{\mathrm{tr}}(\delta,Q)$ which decides the order of magnitude of $E(\delta,Q)$.
\begin{lemma}\label{le error error small} 
Let  $N\geq d+2$ and suppose $\bfF$ is $(N+1)$-times 
continuously differentiable. Assume that $\delta \geq Q^{2\eta -1}$ and $r\geq Q^{\eta -1}$. 
We have
\begin{equation}\label{eq E truncated}
E(\delta,Q) =
E_{\mathrm{tr}}(\delta,Q)
    + O_N(\delta^m J^{n+1} Q^{d+1 } (QrJ)^{-N} 
+Q^{-2N}).   
\end{equation}
\end{lemma}
\begin{proof}
By using Poisson summation in
the $\bfa$ and $q$ variables, we can express $E(\delta,Q)$ as
$$
\delta^m   
    \sum_{(c,\bfj,\bfv')\in \Z\times \cJ \times \Z^{d}} 
    B(\delta \bfj)
    \int_{\R^{d+1}}
    e(y \langle 
    \bfF( \bfx/y ),
    \bfj \rangle
    - \langle 
    \bfx,
    \bfv' \rangle
    - y c)
    \wgood\left(\frac{\bfx}{y}\right) 
    \ome\left(\frac{y}{Q}\right)
    \rd \bfx \rd y.
$$
Next, we change variables twice,
first via $y\mapsto Q y$ and then via 
$\bfx \mapsto Q y \bfx$. 
This transforms the integral above into $Q^{d+1}$ times
\begin{equation}\label{eq intermediate Error}
  \cI(\bfj, \bfv', c):=
\int_{\R^{d+1}}
    y^d e(Qy [\langle 
    \bfF( \bfx ),
    \bfj \rangle
    - \langle 
    \bfx,
    \bfv' \rangle
    - c])
    \wgood(\bfx) 
    \ome(y)
    \,\rd \bfx\, \rd y.  
\end{equation}
Recall $J=Q^\eta/\delta$ 
from \eqref{def J}, and we have assumed that $\mathbf{F}$ is $(N+1)$-times continuously differentiable. 
To proceed, we truncate now the $c$
and $\bfv'$ variables appropriately, making use of integration by parts to get rid of the regimes where these parameters are of large size. 
Let $L^N$ be the differential operator as defined in \eqref{def Lop} with $$\psi(\bfx)=\langle \bfF( \bfx ),\bfj \rangle
- \langle \bfx,\bfv' \rangle,$$ (for a fixed $y$). Note that there is a constant $C_{\bfF}$ larger than one, only depending on $\bfF$, such that for $\Vert \bfv'\Vert_\infty > C_\bfF^{1/2} J$ we have 
\begin{equation}
    \label{eq grad lb}
    \|\nabla \psi\|_\infty \geq (1/2) \Vert \bfv'\Vert_\infty.
\end{equation}
Assume that $\Vert \bfv'\Vert_\infty >
C_\bfF^{1/2} J$.  

Recall the notation from Definition \ref{def typeABterms}. By Lemma \ref{lem partial}, we have that $$L^{\circ N} \wgood = \sum_{\nu=1}^{V} c_{N,\nu} \Omega_{N,\nu},$$
with each $\Omega_{N,\nu}$ of the form
$
P(\tfrac{\nabla \psi}{|\nabla\psi|}) 
\beta_{\mathrm{amp}}
\prod_{l=1}^M \gamma_l$.
Here $P$ is a polynomial of $d$ variables 
(independent of $\wgood$ and $\psi$), $\beta_{\mathrm{amp}}$ 
is of amplitude type with order $j_{\mathrm{amp}}$
for some $j_{\mathrm{amp}}\in \{0,\dots, N\}$, and
each $\gamma_l$ 
is of phase type with order $\kappa_l$, so that \eqref{eq order} is true.

Note that for each $j\in\mathbb{N}$, we have $\|(\wgood)_{j+1}\|_{\infty}\ll r^{-(j+1)}$. 
Combining this with the lower bound \eqref{eq grad lb}, we get 
\begin{equation*}
\|\beta_{\mathrm{amp}}\|_{\infty}\ll (r\Vert \bfv'\Vert_\infty)^{-j_{\mathrm{amp}}} ,    
\end{equation*}
where the implicit constant may depend on $\bfF$. 
Observe that $\|\psi_{j+1}\|_{\infty}\ll Q^{\eta}\delta^{-1}$ for $j\geq 1$,
and thus, using \eqref{eq grad lb} again, we get
$$\|\gamma_l\|_{\infty}\ll Q^{\eta}\delta^{-1}\Vert \bfv'\Vert_\infty^{-\kappa_l-1}\ll \Vert \bfv'\Vert_\infty^{-\kappa_l}.$$
The last inequality follows from our assumption that $\Vert \bfv'\Vert_\infty >  C_\bfF J=C_\bfF Q^{\eta}\delta^{-1}$.
Using \eqref{eq jM sum} and the above observations, we surmise
$\|\Omega_{N, \nu}\|_\infty\ll r^{-N} \Vert \bfv'\Vert_\infty^{-N}.$
Thus
$$Q^{-N}\|L^{\circ N} \wgood\|_{\infty}\ll_{N, d} Q^{-N} r^{-N}\Vert \bfv'\Vert_\infty^{-N},$$
and
\begin{equation}
    \label{eq ibp1}
     \cI(\bfj, \bfv', c)\ll Q^{-N} r^{-N}\Vert \bfv'\Vert_\infty^{-N}. 
\end{equation}

Next, we estimate the integral in (\ref{eq intermediate Error}) for large values of $c$. This time, we integrate by parts in the $y$ variable keeping $\bfx$ fixed. Indeed, the application of Lemma \ref{lem partial}, with $$\psi(y)=y [\langle \bfF( \bfx ), \bfj \rangle- \langle \bfx, \bfv' \rangle-c]),$$ is much simpler as all derivatives of order greater than one of the phase function vanish; and the derivatives of the amplitude $\omega$ are uniformly bounded from above. After possibly enlarging the constant $C_\bfF$, we find that for $|c|\geq C_\bfF^{1/2} \max\{J,\Vert \bfv'\Vert_\infty\}$ and for any $A\in\mathbb{N}$, we have
\begin{equation}
    \label{eq IBP2}
    \cI(\bfj, \bfv', c)\ll_A (|c|Q)^{-A}.
\end{equation}

Recall that $N\geq d+2$. Let $A>0$ be a sufficiently large parameter. We can bound the contribution of triples $(c,\bfj,\bfv')\in \Z^{n+1}$ with $\Vert (c,\bfv')\Vert_\infty > C_{\bfF} J$ as follows
$$
E(\delta, Q)-E_{\mathrm{tr}}(\delta, Q)
\ll \delta^m \sum_{\substack{(c,\bfj,\bfv')\in \Z\times \cJ \times \Z^{d}\\ \Vert (c,\bfv')\Vert_\infty > C_{\bfF} J}}  Q^{d+1}
\cI(\bfj, \bfv', c).
$$
We can split the sum above to express the right hand side as
\begin{equation}
\delta^m Q^{d+1}\sum_{\substack{(c,\bfj,\bfv')\in \Z\times \cJ \times \Z^{d}\\ \Vert (c,\bfv')\Vert_\infty > C_{\bfF} J\\ |c|< C_\bfF^{1/2} \Vert\bfv'\Vert_\infty }}  \cI(\bfj, \bfv', c)+\delta^m Q^{d+1}\sum_{\substack{(c,\bfj,\bfv')\in \Z\times \cJ \times \Z^{d}\\ \Vert (c,\bfv')\Vert_\infty > C_{\bfF} J\\ |c|\geq  C_\bfF^{1/2} \Vert\bfv'\Vert_\infty }} \cI(\bfj, \bfv', c). \label{eq align t}    
\end{equation}
Using \eqref{eq ibp1} and taking $C_{\bfF}$ sufficiently large, the first term in \eqref{eq align t} can be bounded from above by $C_{\bfF}$ times
\begin{align*}
 \delta^m Q^{d+1}\sum_{\substack{(\bfj,\bfv')\in  
 \cJ \times \Z^{d}\\ 
 \Vert \bfv'\Vert_\infty > C_{\bfF}^{1/2} J }}  
 \Vert \bfv'\Vert_\infty \left(Qr \Vert \bfv'\Vert_\infty\right)^{-N}
&\ll \delta^m J^m Q^{d+1-N }r^{-N} \sum_{ \Vert \bfv'\Vert_\infty > C_{\bfF}^{1/2} J }  \Vert \bfv'\Vert_\infty^{1-N} \\
   & \ll \delta^m J^m Q^{d+1-N }r^{-N} J^{d+1-N} \\
   & \ll \delta^m J^{n+1} Q^{d+1 } (QrJ)^{-N}.
\end{align*}
Similarly, we can use \eqref{eq IBP2} to estimate the second term in \eqref{eq align t}, for $A\geq d+2$, as follows
\begin{align*}
\delta^m Q^{d+1} 
\sum_{\substack{(c,\bfj,\bfv')\in 
\Z\times \cJ \times \Z^{d}\\ 
\Vert (c,\bfv')\Vert_\infty > C_{\bfF} J\\ 
|c|\geq  C_\bfF^{1/2} \Vert\bfv'\Vert_\infty }}  
(|c| Q)^{-A}
   & \ll_A \delta^m Q^{d+1} J^m \sum_{|c| > 
   C_{\bfF} J} |c|^{d} (|c| Q)^{-A}\\
   & \ll_A \delta^m J^{m} Q^{d+1-A} 
   J^{d+1-A}\ll_N Q^{-2N}.
\end{align*}
Recalling \eqref{def Phi},
the function in the rectangular brackets
in \eqref{eq intermediate Error} 
is $\Phi_{(-\bfv',\bfj)}$. 
Thus we obtain $$E(\delta,Q) -E_{\mathrm{tr}}(\delta,Q)
    \ll_N  \delta^m J^{n+1} Q^{d+1 } (QrJ)^{-N} 
+Q^{-2N},$$
and this finishes the proof.
\end{proof}

The final lemma of this section establishes that $E(\delta, Q)$ is indeed an acceptable error term.
\begin{lemma}\label{le error small} 
Let $\bfF$ be $(N+1)$-times continuously differentiable. Assume that $\delta \geq Q^{2\eta -1}$ and $K\leq 1$. Further, suppose that
\begin{equation}
     \tag{C3}
    \label{C3}
    \Delta\geq Q^{\eta-1},
\end{equation}
and
\begin{equation}
    \tag{C4}
    \label{C4}
    Q^{\eta-1} K^{-1} \leq r \leq Q^{-2\eta} K\delta.
\end{equation}
Then 
$$
E(\delta,Q) \ll_N  \delta^m Q^{d+1} Q^{n+1-\eta (n+1+N)}.
$$
\end{lemma}
\begin{proof}
We first bound $E_{\mathrm{tr}}(\delta,Q)$.
For any $\bfv \in \cV$ (as defined in \eqref{def Vset}),
and any point $\bfx \in \supp(\wgood)$,
at least one of the inequalities
\begin{equation}\label{eq inequalities to be reversed}
\Vert \Phi_{\bfv}(\bfx)\Vert <  \Delta,
\quad \mathrm{or} \quad 
\Vert \nabla\Phi_{\bfv}(\bfx) \Vert_\infty < K,
\end{equation}
is false.
This motivates us to decompose 
$E_{\mathrm{tr}}(\delta,Q)$ based on the size of 
$\Vert \nabla\Phi_{\bfv}(\bfx) \Vert_\infty$ with respect to $K$. 
Let $b:\mathbb{R}^{d}\to [0,1]$ be a smooth weight, 
supported on $\{\bfz\in\mathbb{R}^d:\|\bfz\|_2<1\}$ and 
identically equal to $1$ on 
$\{\bfz\in\mathbb{R}^d:\|\bfz\|_2\leq 1/2\}$. 
For each $\bfv=(\bfj, \bfv')\in \cV$ and 
$c\in \mathbb{Z}$ with $0 \leq |c| \leq C_{\mathbf{F}}J$, we split
\[\int_{\R^{d+1}}
    y^d e(y Q 
    [\Phi_{\bfv}(\bfx) - c] )
    \wgood(\bfx) 
    \ome(y)
    \rd \bfx \rd y
    =  I(\bfv, c)+ II(\bfv, c),
\]
with
\[I(\bfv, c):=\int_{\R^{d+1}}
    y^d e(y Q 
    [\Phi_{\bfv}(\bfx) - c] )
    \wgood(\bfx) 
    \ome(y) b(K^{-1}\nabla \Phi_{\bfv} (\bfx))
    \rd \bfx \rd y\]
    and
\[II(\bfv, c):= \int_{\R^{d+1}}
    y^d e(y Q 
    [\Phi_{\bfv}(\bfx) - c] )
    \wgood(\bfx) 
    \ome(y) (1-b(K^{-1}\nabla \Phi_{\bfv} (\bfx)))
    \rd \bfx \rd y.\]
We deal with the second integral first. 
The amplitude function in $II(\bfv, c)$ is supported on the set where $\|\nabla \Phi_{\bfv} (\bfx)\|_2 \geq K/2$, which lets us integrate by parts in the $\bfx$ variable, keeping $y$ fixed.  Let $L^N$ be the differential operator as defined in \eqref{def Lop} with $\psi=\Phi_{\bfv}$. Set $$h=\wgood\cdot [1-b(K^{-1}\nabla \Phi_{\bfv})].$$ Lemma \ref{lem partial} tells us that $L^{\circ N} h = \sum_{\nu=1}^{V} c_{N,\nu} h_{N,\nu}$
with each $h_{N,\nu}$ of the form
$$
P(\tfrac{\nabla \psi}{|\nabla\psi|}) 
\beta_{\mathrm{amp}}
\prod_{l=1}^M \gamma_l.$$
As before, $P$ is a polynomial of $d$ variables 
(independent of $h$ and $\psi$), $\beta_{\mathrm{amp}}$ 
is of amplitude type and order $j_{\mathrm{amp}}$,
for some $j_{\mathrm{amp}}\in \{0,\dots, N\}$, while
each $\gamma_l$ 
is of phase type and order $\kappa_l$
so that \eqref{eq order} is satisfied.

A few crucial observations are now in order. Let $j\in \mathbb{Z}_{\geq 0}$. We shall use the notation from Definition \ref{def typeABterms}.

(i) Since $r\leq \delta K Q^{-2\eta}$,
we observe that
$$
\|h_{j}\|_{\infty}\ll (K^{-1}Q^{\eta}\delta^{-1}+r^{-1})^{j}\ll r^{-j}.
$$
Combining this with the assumed lower bound on $\|\nabla \Phi_\bfv\|_\infty$, we get $\|\beta_{\mathrm{amp}}\|_{\infty}\ll (rK)^{-j_{\mathrm{amp}}}.$

(ii) Further, 
$\|\psi_{j+1}\|_{\infty}\ll Q^{\eta}\delta^{-1}$.
By using the lower bound on 
$\|\nabla \psi\|_\infty=\|\nabla 
\Phi_\bfv\|_\infty$ again, we get
$$
\|\gamma_l\|_{\infty}\ll 
Q^{\eta} \delta^{-1}K^{-(\kappa_l+1)}.
$$
Using \eqref{eq order}, \eqref{eq jM sum} and the above observations, we surmise
\begin{align*}
\|h_{N, \nu}\|_\infty\ll (rK)^{-j_{\mathrm{amp}}}\prod_{l=1}^M (Q^{\eta}\delta^{-1} K^{-(\kappa_l+1)})&=r^{-j_{\mathrm{amp}}}Q^{M\eta}K^{-N}(\delta K)^{-M}\\
&\leq r^{-j_{\mathrm{amp}}} K^{-N} (Q^{-\eta}\delta K)^{-N+j_{\mathrm{amp}}}.    
\end{align*}
Recall that $(Q^{-\eta}\delta K)^{-1}\ll 
r^{-1}$ and $(rK)^{-1} \leq Q^{1-\eta}$ 
by our assumptions, which lets us conclude
$$\|h_{N, \nu}\|_\infty\ll  r^{-j_{\mathrm{amp}}}K^{-N}r^{-N+j_{\mathrm{amp}}}\ll Q^{(1-\eta)N}.$$
Thus $Q^{-N}\|L^{\circ N} h\|_{\infty}\ll_{N, d} Q^{-N\eta},$
and integrating by parts (with respect to $\bfx$) $N$ times yields
$$ II(\bfv, c) \ll Q^{-N\eta}.$$

We now bound $I(\bfv, c)$. Observe that the support of $ \wgood\cdot b(K^{-1}\nabla \Phi_{\bfv} )$ is contained in the complement of $\mathfrak{S}_\bfp$,  where $\|\nabla \Phi_{\bfv}\|_\infty <K$. 
 By the definition of $\mathfrak{S}_\bfp$, we conclude that for any $c\in\mathbb{Z}$, $$|\Phi_{\bfv}(\bfx)-c| \geq \|\Phi_{\bfv}(\bfx) \|\geq \Delta\geq Q^{\eta-1}$$ on this set. Thus, we integrate by parts in the $y$ variable (noting that $|d^{j}\omega(y)|\ll_j 1$ and $\partial^{j+1}_y\psi =0 $ for any $j\in\mathbb{Z}_{\geq 1}$) to conclude that 
 $$I(\bfv, c)\ll_{N} Q^{-N \eta}$$
 for any $N\in \N$.
Summing up the contributions from $\bfv\in \mathcal{V}$ and $0<|c|<C_{\bfF}J$, we get
$$E_{\mathrm{tr}}(\delta, Q)\ll_N \delta^m 
Q^{d+1}
    \sum_{0 \leq \vert c\vert \leq C_\bfF J} 
    \sum_{\bfv\in \cV} Q^{-N\eta}
    \ll_N \delta^m Q^{d+1} J^{n+1}Q^{-N\eta}.$$
In view of \eqref{eq E truncated} and the above, we can conclude
\begin{equation}
    \label{eq Error est}
    E(\delta, Q)\ll_N \delta^m Q^{d+1} J^{n+1} 
    Q^{-N\eta}+\delta^m Q^{d+1 } J^{n+1} (QrJ)^{-N} 
    +Q^{-2N}.
\end{equation}
Note that \eqref{def J} and the assumed lower bound on $\delta$ together imply that $J\leq Q^{1-\eta}$. Moreover, since $r\geq Q^{\eta -1}$ by our assumption, we have 
$QrJ \gg Q^{2\eta} \delta^{-1} \gg Q^{2 \eta}$.
Thus,
$$
E(\delta, Q)\ll_N  
\delta ^m Q^{d+1} J^{n+1} Q^{-N\eta} 
\ll_N \delta ^m Q^{d+1} Q^{(1-\eta)(n+1)} Q^{-N\eta}.
$$
This gives the required estimate.
\end{proof}
\section{Proof of Theorems \ref{thm main asymp bounds}-\ref{thm main upper bounds}}\label{parameters}
\subsection{Choice of parameters}
\label{subsec par choice}
We set
$$J=Q^{\eta}\delta^{-1},\quad T=2C_\bfF Q^{\eta} \delta^{-1}, 
\quad \Delta=Q^{4\eta-1},
\quad K=\delta^{-\frac{1}{2}}
Q^{2\eta-\frac{1}{2}},
\quad 
r=\delta^{\frac{1}{2}}
Q^{-\frac{1}{2}}.$$
Assume that
$$\eta \leq 1/8,\qquad Q^{4\eta -1}\leq \delta <1/2.$$
It is easily checked that with the above choice, 
the conditions \eqref{C1}-\eqref{C4} 
are all satisfied as well as the conditions
$$r\geq Q^{\eta -1},\qquad K\leq 1.$$
Observe that 
\begin{equation}\label{eq measure comp}
    K\Delta T^{n-1}\ll 
\delta^{-\frac{1}{2}}
Q^{2\eta - \frac{1}{2}} Q^{4\eta -1} (Q^{\eta} \delta^{-1})^{n-1}\ll \delta^{-(n-\frac{1}{2})}
Q^{-\frac{3}{2}+ (n+5)\eta}. 
\end{equation}

We now shortly compare the choice of our parameters with those in work of Beresnevich and Yang \cite{BY}. In the notation above, they choose $T$
 and $\Delta$ of comparable size, but their value of $K$, say $K_{BY}$ is roughly of size $K_{BY}\asymp \delta^{-1} Q^{-\frac{1}{2}}$, and hence typically larger than our choice of $K$. This is one advantage of our method, that we can take the exceptional `bad' set, also called the 'special part' in \cite{BY}, to be of smaller size.
 
\subsection{The main proofs}
With the above choice in place, we are now ready to prove our main results on counting rational points. We begin by establishing the lower bounds. 

\begin{proof}[Proof of Theorem \ref{thm main lower bounds}]
Recall the notation from Lemma \ref{le decomp}.
By \eqref{eq basic lower bound} 
and \eqref{eq Fourier decomp}, we see 
$$
N^{\Omega}(\delta,Q)
    \geq    N^{\wgood}(\delta,Q)=
M(\delta, Q)
+ E(\delta,Q)
+ O_A(Q^{-A}),
$$
for any $A>0$.
In view of \eqref{eq measure comp} and the assumption on $\delta$ in Theorem \ref{thm main lower bounds},
Theorem \ref{thm quant non div} 
implies that $$\mu(\fS_{2\bfp})=O(Q^{-\frac{\eta}{d(2l-1)(n+1)}}).$$
Note that for $n\geq 2$ and $0<\eta \leq \frac{1}{8}$ we have
$$Q^{4\eta -1} \leq Q^{-\frac{3}{2n-1}+\frac{2n+12}{2n-1}\eta}.$$
As a result, Lemma \ref{le zero mod comp} (more specifically, \eqref{eq main term computed fully}) 
tells us
\begin{align*} 
M(\delta,Q)&= \delta^m Q^{d+1}\cdot \left(1+O(Q^{-\frac{\eta}{d(2l-1)(n+1)}})\right)
(\widehat{w}^m 
\widehat{\Omega})(0)
\int_{\R} x^d\omega(x) 
\rd x 
\\
&= c_\bft \delta^m Q^{d+1} +O\left(\delta^m Q^{d+1-\frac{\eta}{d(2l-1)(n+1)}}\right).
\end{align*}
By assumption, $\bfF$ is at least $\frac{n+1}{\eta}$ times continuously differentiable. 
Hence Lemma \ref{le error small}, 
upon choosing $N:=\lceil \frac{n+1}{\eta}\rceil $ ensures $E(\delta,Q)=O\left(\delta^mQ^{d+1-\eta}\right)$.
Thus we obtain the desired lower bound.
\end{proof}
The proof of the upper bounds proceeds in a similar way.
\begin{proof}[Proof of Theorem \ref{thm main upper bounds}]
By the third item of 
Lemma \ref{le smooth cut off} and 
Lemma \ref{le rationals in a ball},
we infer that 
\begin{equation}\label{eq bounding bad part}
N^{\wbad}(\delta,Q) 
\ll \sum_{q\asymp Q} 
t_r \cdot (qr)^d
\ll Q^{d+1}
\mu_d(\fS_{2\bfp}).
\end{equation}
Using Theorem \ref{thm quant non div} 
and the estimate \eqref{eq measure comp}, we deduce
\begin{equation}\label{eq N bad bound}
N^{\wbad}(\delta,Q) 
\ll
Q^{d+1}
(\delta^{-(n-\frac{1}{2})}
Q^{-\frac{3}{2}+ (n+5)\eta} )^{\frac{1}{d(2l-1)(n+1)}}
.
\end{equation}
By the definition of $l$-nondegeneracy, 
$l\geq 2$. 
If $n\geq 2$ and $\eta \leq \frac{1}{2n+10}$, then
$$Q^{\frac{-3+(2n+10)\eta}{2md(2l-1)(n+1)+2n-1}}\geq Q^{4\eta -1}.$$
Further, by the proof of Theorem \ref{thm main lower bounds},
we know $ N^{\wgood}(\delta,Q) 
\ll \delta^m Q^{d+1}$ for 
\begin{equation}\label{rangedelta2}
\delta \geq Q^{\frac{-3+(2n+10)\eta}{2md(2l-1)(n+1)+2n-1}}.
\end{equation}
Hence, for those values of $\delta$ we obtain the upper bound
$$N^{\Omega}(\delta,Q)\ll \delta^m Q^{d+1}  + Q^{d+1}
(\delta^{-(n-\frac{1}{2})}
Q^{-\frac{3}{2}+ (n+5)\eta} )^{\frac{1}{d(2l-1)(n+1)}}.$$
Note that for $\delta$ in the range \eqref{rangedelta2}, the first term dominates. By monotonicity of the counting function $N^{\Omega}(\delta,Q)$, we find that
$$N^{\Omega}(\delta,Q)\ll \delta^m Q^{d+1} + (Q^{\frac{-3+(2n+10)\eta}{2md(2l-1)(n+1)+2n-1}})^m Q^{d+1} ,$$
for all $0< \delta \leq 1/2$. 
We can rewrite this as
$$N^{\Omega}(\delta,Q)\ll \delta^m Q^{d+1} + Q^{d+1- \frac{3m-m(2n+10)\eta}{2md(2l-1)(n+1)+2n-1}}   ,$$
which implies the desired result.
\end{proof}
Finally, we combine the upper 
and lower bounds to deduce
asymptotics for $N^{\Omega}(\delta,Q)$.
\begin{proof}[Proof of Theorem \ref{thm main asymp bounds}]
Assume that $\bfF$ is $N$ 
times continuously differentiable, 
and recall
$$N^{\Omega}(\delta,Q)
    =    N^{\wgood}(\delta,Q) + N^{\wbad}(\delta,Q).$$
By the proof of Theorem \ref{thm main lower bounds} and Theorem \ref{thm main upper bounds}, and by using Lemma \ref{le error small} we infer 
\begin{align*}
N(\delta,Q) =&c_\bft \delta^m Q^{d+1} +O\left( Q^{d+1}(\delta^{-(n-\frac{1}{2})}
Q^{-\frac{3}{2}+ (n+5)\eta} )^{\frac{1}{d(2l-1)(n+1)}} \right) \\&+ O\left( \delta^mQ^{d+1+n+1-\eta(n+N)}  \right).     
\end{align*}
If we assume that $n\geq 2$ and $N \geq \frac{n+2}{\eta}$, then this simplifies to
\begin{equation}\label{almost asymp}
N(\delta,Q) =c_\bft \delta^m Q^{d+1} +O\left( Q^{d+1}(\delta^{-(n-\frac{1}{2})}
Q^{-\frac{3}{2}+ (n+5)\eta} )^{\alpha} \right)
\end{equation}
where $\alpha:= 1/(d(2l-1)(n+1))$.
Now we deduce the ``in particular'' part 
of the theorem. 
To this end we notice that the 
scale, say, $\delta_0$
for which the expected main term 
$\delta^m Q^{d+1}$
and the $O$-term are of the same magnitude
is given by
$$
\delta_0 :=
Q^{\alpha \frac{-3/2+ (n+5)\eta}
{m+\alpha (n-1/2)}}.
$$
The assumption $\eta< 3/(2(n+5))$ ensures 
$-3/2+ (n+5)\eta<0$. Hence, we see
$$\delta_0 < 
\delta_{\mathrm{simp}}:=
Q^{\alpha \frac{-3/2+ (n+5)\eta}
{m+\alpha (n-1/2)}}
Q^{\frac{\epsilon}{\alpha (n-1/2)}}.
$$
Observing $\delta_{\mathrm{simp}}/\delta_{0}
\gg Q^{\frac{\epsilon}{\alpha (n-1/2)}}$, we 
write
$\delta\in (\delta_{\mathrm{simp}},1/2)$ 
in the form $\delta= \xi \delta_{0}$
with $\xi \gg Q^{\frac{\epsilon}{\alpha (n-1/2)}}$.
Plugging this factorisation into
\eqref{almost asymp} yields
$$
N(\delta,Q) =c_\bft \xi^m 
\delta_0^m Q^{d+1} +O\left( \delta^m Q^{d+1-\epsilon}
\right).
$$
\end{proof}

\section{Applications to Diophantine Approximation: Proofs}\label{proofsapplications}

\subsection{Upper Bounds for Hausdorff dimension}

\begin{proof}[Proof of Theorem \ref{thm haus s zero}]
Since $\psi$ is monotonic, by a slight generalisation of Cauchy's condensation test, \eqref{eq Hmeas gset conv} being convergent is equivalent to 
\begin{equation}\label{eq dyad Hmeas gset}
 \sum_{t=1}^\infty \
 e^{(n+1)t}\Big(\frac{\psi(e^t)}{e^t}\Big)^{s+m}<\infty\,.
\end{equation}
We assume without loss of generality that
$\mathcal{M}$ is in a normalised Monge form, that is
$$
\mathcal{M}=\{(\bfx,\bfF(\bfx)), \bfx\in 
\sU_d\},
$$
where $\sU_d = \sB(\bzero,1)$ is the unit ball. 
Fix $\bfx_0\in\sU_d$ such that $\bfF$ is $l$--nondegenerate at $\bfx_0$. 
Recall the set $\fS_{\bfp, T}=\fS_{\Delta, K, T}$ 
as defined in \eqref{def non osc set}, and let $\epsilon>0$ be as in \eqref{eq Hmeas bset conv}. Keeping the choice of parameters from 
\S \ref{subsec par choice} in mind, we let
$$\Delta_0=\Delta_0(Q):=Q^{4\epsilon-1},\,\, T_0=T_0(\delta,Q):=2C_{\bfF}Q^{\epsilon}\delta^{-1},$$ 
$$ K_0(\delta, Q)=K_0:=\delta^{-\frac{1}{2}}Q^{2\epsilon-\frac{1}{2}},$$
and
$$\fM(\delta, Q):=\fS_{\Delta_0, K_0, T_0}.$$
Observe that if the series in \eqref{eq Hmeas bset conv} converges for some $\epsilon_0$, then it also converges for all $\epsilon \leq \epsilon_0$. Hence we may assume without loss of generality that $\epsilon <1/8$. If $\delta$ satisfies the inequalities
$$Q^{4\epsilon-1}\leq \delta <1/2,$$
then by Theorem \ref{thm quant non div}, 
there exists a ball $B_0\subseteq \sU_d$, 
centred at $\bfx_0$, such that
\begin{equation}
    \label{eq bset est}
    \mu_d(\fM(\delta, Q)\cap B_0)\ll 
    \left(\delta^{-(n-\frac{1}{2})}Q^{-\frac{3}{2}+(n+5)\epsilon}\right)^{\frac{1}{d(2l-1)(n+1)}}
    \mu_d(B_0).
\end{equation}

Write $\bff(\bfx)=(\bfx,\bfF(\bfx))$ and for $A\subseteq \mathscr{U}_d$, let 
$$\cR(A, \delta, Q):=\left\{(\bfa, q)\in \mathbb{Z}^{n+1}:e^{-1}Q \leq q<Q,\, \inf_{\bfx\in A}\left\|\bff(\bfx)-\frac{\bfa}{q}\right\|_{\infty}<\frac{\delta}{Q}\right\}.$$
For a vector $\bfa\in \Z^n$ we write $\bfa'$ for its first $d$ coordinates.\par
By \eqref{eq excep Hausd}, it suffices to show that
$$
\cH^s\big(\big\{\bfx\in B_0:(\bfx,\bfF(\bfx))\in\cS_n(\psi)\big\}\big)=0.
$$ 
For each $t\geq T$ consider the set
$$\widetilde{A}_t:= \fM(e\psi(e^{t-1}), e^t).$$
Note that the convergence of the series \eqref{eq Hmeas bset conv} for $s\leq d$ implies that for $t$ sufficiently large, we have
$$e^{t(4\epsilon -1)}\leq e \psi(e^{t-1}).$$
We now apply Lemma \ref{le smooth cut off} with $r=(\psi(e^{t-1})e^{-(t-1)})^{\frac{1}{2}}$, and observe that the set $\widetilde{A}_t$ can be covered by
$$
\ll \left(\psi(e^{t-1}) e^{-(t-1)}\right)^{-\frac{d}{2}}\left(\psi(e^{t-1})^{n-\frac{1}{2}}e^{(\frac{3}{2}-(n+5)\epsilon)(t-1)}\right)^{-\alpha}
$$
balls of radius $2r=2\left(\psi(e^{t-1}) e^{-(t-1)}\right)^{\frac{1}{2}}$. Write $A_t$ for the union of these balls of radius $2r$.

By \cite[Lemma 1.4]{BY}, for any $T\ge1$, we can express
$$\big\{\bfx\in B_0:(\bfx,\bfF(\bfx))\in\cS_n(\psi)\big\}=\bigcup_{t\ge T} \left(A_t \cup B_t\right),$$
where
$$B_t:=\bigcup_{(\bfa,q)\in
\cR(B_0\setminus A_t;
e\psi(e^{t-1}),e^t)}
\left\{\bfx\in B_0: \left\| \bfx -\frac{\bfa'}{q} 
\right\|_{\infty}<\frac{\psi(e^{t-1})}{e^{t-1}}
\right\}.$$
We now use the notation as in section \ref{sublevelFA}, and w.l.o.g. we may assume that $\Omega$ is identically equal to one on $B_0$. If we are given a vector $(\bfa,q)\in
\cR(B_0\setminus A_t;
e\psi(e^{t-1}),e^t)$, then we claim that $W_r\left(\frac{\bfa'}{q}\right)=0$. Let $\bfy\in B_0\setminus A_t$ such that $$\left\Vert \bff(\bfy)-\frac{\bfa}{q}\right\Vert_\infty < \frac{\psi(e^{t-1})}{e^{t-1}}.$$

If $\frac{\bfa'}{q}$ is contained in the support of $W_r$ and more precisely is contained in one of the balls of radius $r$ described in Lemma \ref{le smooth cut off}, say with center $\bfc$, then
$$\left\Vert \bfc - \bfy\right\Vert_2 \leq \left\Vert \bfc-\frac{\bfa'}{q}\right\Vert_2 + \left\Vert \frac{\bfa'}{q}-\bfy\right\Vert_2 \leq r+ d  \frac{\psi(e^{t-1})}{e^{t-1}}.$$
Note that for $t$ sufficiently large, the right hand side is $<2r$, a contradiction to $\bfy\in B_0\setminus A_t$, and we find that $W_r\left(\frac{\bfa'}{q}\right)=0$.

By combining Lemmas \ref{le decomp} and \ref{le error small} with the estimate \eqref{eq main term computed fully}, we deduce 
that 
$$\sharp\cR(B_0\setminus A_t;
e\psi(e^{t-1}),e^t) \ll \psi(e^{t-1})^{m}e^{(d+1)(t-1)}.$$
Hence the set $B_t$ is contained in the union of $\ll\psi(e^{t-1})^{m}e^{(d+1)(t-1)}$ balls of radius $\ll \frac{\psi(e^{t-1})}{e^{t-1}}$. We conclude that

$$
\cH^s\left(\big\{\bfx\in B_0:(\bfx,\bfF(\bfx))\in\cS_n(\psi)\big\}\right)\ll \sum_{t\ge T}\psi(e^{t-1})^{m}e^{(d+1)(t-1)} \;\cdot\;\left(\frac{\psi(e^{t-1})}{e^{t-1}}\right)^s+
$$
\begin{equation}\label{eq2.15}
+\sum_{t\ge T}\left(\psi(e^{t-1}) e^{-(t-1)}\right)^{\frac{s-d}{2}}\left(\psi(e^{t-1})^{n-\frac{1}{2}}e^{(\frac{3}{2}-(n+5)\epsilon)(t-1)}\right)^{-\alpha}\,.
\end{equation}
The first sum can be rewritten as
$$
\sum_{t\ge T}e^{(n+1)(t-1)} \left(\frac{\psi(e^{t-1})}{e^{t-1}}\right)^{s+m},
$$
which converges to zero by \eqref{eq dyad Hmeas gset} as $T\rightarrow\infty$. Similarly, the second sum also tends to zero, this time due to \eqref{eq Hmeas bset conv}. We can thus conclude that $$\cH^s\left(\big\{\bfx\in B_0:(\bfx,\bfF(\bfx))\in\cS_n(\psi)\big\}\right)=0.$$
\end{proof}

\begin{proof}[Proof of Corollary \ref{cor lbound dim}]
Let $s>\frac{n+1}{\tau+1}-\textrm{ codim }\cM$ and $\psi(q)=q^{-\tau}$. Since $\cM$ is $l$--nondegenerate everywhere except possibly on a set of Hausdorff dimension less than $s$, \eqref{eq excep Hausd} is automatically true. Further, a straightforward calculation using the lower bound on $s$ shows that \eqref{eq Hmeas gset conv} is convergent. 

Next, we show that there exists a choice of $\epsilon$ for which \eqref{eq Hmeas bset conv} is true.
For any $\epsilon>0$, we can express
\begin{equation}
    \label{eq beta series}
    \sum_{t=1}^{\infty}\left(\frac{\psi(e^t)}{e^{t}}\right)^{\frac{s-d}{2}}(\psi(e^t)^{n-\frac{1}{2}}e^{(\frac{3}{2}-(n+5)\epsilon)t})^{-\alpha}=\sum_{t=1}^{\infty}e^{\beta t},
\end{equation}
where
\begin{equation}
    \label{eq beta1}
\beta=-\frac{(\tau+1)(s-d)}{2}+\alpha\tau\left(n-\frac{1}{2}\right)-\alpha\left(\frac{3}{2}-(n+5)\epsilon\right).
\end{equation}
Observe that due to our assumption on $s$, we have
$$-\frac{(\tau+1)(s-d)}{2}< \frac{n\tau-1}{2}.$$
Plugging the above into \eqref{eq beta1} gives
$$\beta < \tau\left[\frac{n}{2}+\alpha\left(n-\frac{1}{2}\right)\right]-\alpha\left(\frac{3}{2}-(n+5)\epsilon\right)-\frac{1}{2}.$$
Thus, a sufficient condition for the series in \eqref{eq beta series} to converge is that 
$$\tau\leq \frac{3\alpha+1-2(n+5)\epsilon}{\alpha(2n-1)+n}.$$
Due to $\eqref{eq tau cond}$, we can always choose an $\epsilon>0$ such that the above condition, and therefore \eqref{eq dyad Hmeas gset}, is satisfied for the given $\tau$.

We can thus use Theorem \ref{thm haus s zero} to conclude that $\cH^s(\cS_n(\tau)\cap\cM)=0$, and consequently $\dim (\cS_n(\tau)\cap\cM)\le s$. Since $s>\frac{n+1}{\tau+1}-\textrm{ codim }\cM$ is arbitrary we conclude that 
$$\text{dim}(\mathcal{M}\cap\mathcal{S}_n(\tau))\leq \frac{n+1}{\tau+1}-\text{codim }\mathcal{M}.$$
To upgrade the above to an equality, we use the lower bound \eqref{eq BLVV17} provided by \cite[Theorem 1]{BLVV17}
to deduce that we also have
$$\text{dim}(\mathcal{M}\cap\mathcal{S}_n(\tau))\geq \frac{n+1}{\tau+1}-\textrm{ codim }\cM,$$
for 
$$\tau<\frac{3\alpha+1}{(2n-1)\alpha+n}< \frac{1}{n-d}.$$
This finishes the proof.
\end{proof}

\subsection{Proof of Corollary \ref{cor dio exp}}
As the arguments to complete the proof of Corollary \ref{cor dio exp} are rather standard, we only sketch here the details.\par 

The following proposition, proved in \cite[\S 2.3]{BY} using standard analysis, allows us to fix the nondegeneracy parameter $l$ to be $n$ in the special case of curves.
\begin{proposition}[\cite{BY}]
Let $\bfF:\mathscr{U}\to\R^n$ be $l$-nondegenerate at $x_0\in\mathscr{U}$, where $\mathscr{U}$ is an interval in $\R$. Then there is an interval $B_0$ centred at $x_0$ and a countable subset $S\subset B_0$ such that $\bfF$ is $n$-nondegenerate at every point $x\in B_0\setminus S$.  
\end{proposition}
Applying Corollary \ref{cor lbound dim} to nondegenerate curves with $l=n$, then yields the desired result.

\subsubsection*{Acknowledgements}
R. S. was
supported by the Deutsche 
Forschungsgemeinschaft 
(DFG, German Research Foundation) 
under Germany’s Excellence Strategy 
- EXC-2047/1 
– 390685813 as well as SFB 1060.
N. T. was supported by
a Schr\"{o}dinger Fellowship 
of the Austrian Science Fund (FWF):
project J 4464-N. The authors are grateful to the Mathematics department of the University of G\"ottingen for its hospitality, and the organizers of the workshop “Diophantine Geometry, Groups, and Number Theory” for a productive environment.

\bibliographystyle{amsbracket}

\begin{thebibliography}{18}

\bibitem{ACPS}
T. Anderson, L. Cladek, M. Pramanik, 
and A. Seeger. 
\emph{Spherical means on the Heisenberg group: Stability of a maximal function estimate.}
Journal d'Analyse Math\'{e}matique (2021): 
1--28.


\bibitem{BB20}
D. Badziahin and Y. Bugeaud.
\newblock On simultaneous rational approximation to a real number and its
  integral powers, {II}.
\newblock {\em New York J. Math.}, 26:362--377, 2020.


\bibitem{BM90}
   V. V. Batyrev and Yu. I. Manin.
   \newblock Sur le nombre des points rationnels de hauteur born\'{e} des
   vari\'{e}t\'{e}s alg\'{e}briques.
   \newblock{\em Math. Ann.},
   286: 27--43, 1990. 
   


\bibitem{Bers12}
V. Beresnevich.
\newblock Rational points near manifolds and metric {D}iophantine
  approximation.
\newblock {\em Ann. of Math. (2)}, 175(1):187--235, 2012.



\bibitem{BDV06}
V. Beresnevich, D. Dickinson, and S. Velani.
\newblock Measure theoretic laws for lim sup sets
\newblock {\em Memoirs of the American Mathematical Society}
179.846 (2006): 1--98.


\bibitem{BDV07}
V. Beresnevich, D. Dickinson, and S. Velani.
\newblock Diophantine approximation on planar curves and the distribution of
  rational points.
\newblock {\em Ann. of Math. (2)}, 166(2):367--426, 2007.
\newblock With an Appendix II by R. C. Vaughan.




\bibitem{BK22}
V. Beresnevich and D. Kleinbock:
\emph{Quantitative nondivergence and Diophantine approximation on
   manifolds.}
In: Dynamics, geometry, number theory---the impact of Margulis on
      modern mathematics, Univ. Chicago Press, Chicago, IL 2022: 303--341.
   


\bibitem{BKM}
V. Bernik, D. Kleinbock, and G. Margulis:
\emph{Khintchine-type theorems on manifolds: the convergence case for standard and multiplicative versions.}
International Mathematics Research Notices 2001.9 
(2001): 453--486.

\bibitem{BLVV17}
V. Beresnevich, L. Lee, R.~C. Vaughan, 
and S. Velani.
\newblock Diophantine approximation on manifolds and lower bounds for
  {H}ausdorff dimension.
\newblock {\em Mathematika}, 63(3):762--779, 2017.

\bibitem{BRV16}
V. Beresnevich, F. Ram\'{\i}rez, 
and S. Velani.
\newblock Metric {D}iophantine approximation: aspects of recent work.
\newblock In {\em Dynamics and analytic number theory}, volume 437 of {\em
  London Math. Soc. Lecture Note Ser.}, pages 1--95. Cambridge Univ. Press,
  Cambridge, 2016.



\bibitem{BVVZ17}
V. Beresnevich, R.~C. Vaughan, 
S. Velani, and E. Zorin.
\newblock Diophantine approximation on manifolds and the distribution of
  rational points: contributions to the convergence theory.
\newblock {\em Int. Math. Res. Not. IMRN}, (10):2885--2908, 2017.

\bibitem{BVVZ21}
V.~Beresnevich, R.~C. Vaughan, S.~Velani, and E.~Zorin.
\newblock Diophantine approximation on curves and the distribution of rational
  points: contributions to the divergence theory.
\newblock {\em Adv. Math.}, 388:Paper No. 107861, 33, 2021.

\bibitem{BY}
V. Beresnevich and L. Yang. {Khintchine's Theorem and Diophantine Approximation on Manifolds}. \textit{Acta Math.} (to appear). arXiv:2105.13872

\bibitem{BZ}
V. Beresnevich and E. Zorin.
\newblock Explicit bounds for rational points near planar curves and metric
  {D}iophantine approximation.
\newblock {\em Adv. Math.}, 225(6):3064--3087, 2010.




\bibitem{Besicovich-1934}
A.S. Besicovich.
\newblock Sets of fractional dimensions (iv): on rational approximation to real
  numbers.
\newblock {\em J. Lond. Math. Soc.}, 9:126--131, 1934.

\bibitem{Browning}
T. D. Browning.
\newblock Quantitative Arithmetic of Projective Varieties. 
\newblock{\em Progr. Math.} 277, Birkhaeuser, Basel, 2009. 

\bibitem{BHS06}
T. D. Browning, D. R. Heath-Brown, and P. Salberger.
\newblock Counting rational points on algebraic varieties.
\newblock{\em Duke Math. J.} 132 545--578, 2006.

\bibitem{BL05}
Y. Bugeaud and M. Laurent.
\newblock Exponents of {D}iophantine approximation and {S}turmian continued
  fractions.
\newblock {\em Ann. Inst. Fourier (Grenoble)}, 55(3):773--804, 2005.

\bibitem{BL07}
Y. Bugeaud and M. Laurent.
\newblock Exponents of {D}iophantine approximation.
\newblock In {\em Diophantine geometry}, volume~4 of {\em CRM Series}, pages
  101--121. Ed. Norm., Pisa, 2007.

  



\bibitem{HuangRPAdv}
J.-J. Huang.
\newblock Rational points near planar curves and {D}iophantine approximation.
\newblock {\em Adv. Math.}, 274:490--515, 2015.

\bibitem{Huang19}
J.-J. Huang.
\newblock Integral points close to a space curve.
\newblock {\em Math. Ann.}, 374(3-4):1987--2003, 2019.

\bibitem{Huang20}
J.-J. Huang.
\newblock The density of rational points near hypersurfaces.
\newblock {\em Duke Math. J.}, 169(11):2045--2077, 2020.

\bibitem{HuangE}
J.-J. Huang.
\newblock Extremal affine subspaces and Khintchine-Jarník type theorems.
\newblock {\em arXiv:2208.04255}, preprint 2022.

\bibitem{JJHJL}
J.-J. Huang and J.~Liu.
\newblock Simultaneous approximation on affine subspaces.
\newblock {\em Int. Math. Res. Not. IMRN}, 2019.

\bibitem{Hux94}
M.~N. Huxley.
\newblock The rational points close to a curve.
\newblock {\em Ann. Scuola Norm. Sup. Pisa Cl. Sci. (4)}, 21(3):357--375, 1994.

\bibitem{Hoermander: The Analysis}
L. H\"{o}rmander: 
\emph{The Analysis of linear partial 
differential operators I, 
Distribution theory and Fourier Analysis}. 
Springer, 1990.





\bibitem{Jarnik-1929}
A.V. Jarnik.
\newblock Diophantischen {Approximation}en und {Hausdorff}sches {M}ass.
\newblock {\em Mat. Sbornik.}, 36:371--382, 1929.

\bibitem{Jarnik-1931}
A.V. Jarnik.
\newblock \"{U}ber die simultanen diophantischen {Approximationen}.
\newblock {\em Math. Z.}, 33:505--543, 1931.











\bibitem{KM 1998}
D. Kleinbock,  and G. Margulis.
\newblock
Flows on homogeneous spaces and Diophantine approximation on manifolds.
\newblock 
{\em Annals of Mathematics} (1998): 339--360.

\bibitem{KM 1999}
D. Kleinbock,  and G. Margulis.
\newblock 
Logarithm laws for flows on homogeneous spaces.
\newblock {\em Inventiones mathematicae} 138 (1999): 451--494.

\bibitem{Mat}
P.~Mattila.
\newblock {\em Geometry of sets and measures in Euclidean spaces}, volume~44 of
  {\em Cambridge Studies in Advanced Mathematics}.
\newblock Cambridge University Press, 1995.

\bibitem{Sal07}
P. Salberger. 
\newblock On the density of rational and integral points on algebraic varieties. 
\newblock{\em J. Reine Angew. Math.}, 606:123–147, 2007.

\bibitem{SY22}
D. Schindler and S. Yamagishi.
\newblock Density of rational points near/on compact manifolds with certain
  curvature conditions.
\newblock {\em Adv. Math.}, 403:Paper No. 108358, 36, 2022.

\bibitem{Serre}
J.-P. Serre.
\newblock Lectures on the Mordell-Weil Theorem, 3rd ed..
\newblock{\em Aspects of Math.}, Vieweg,
Braunschweig, 1997.

\bibitem{Sim18}
D. Simmons.
\newblock Some manifolds of {K}hinchin type for convergence.
\newblock {\em J. Th\'{e}or. Nombres Bordeaux}, 30(1):175--193, 2018.

\bibitem{ST}
R. Srivastava and N. Technau, 
\emph{Density of rational points near flat/rough hypersurfaces}, arXiv:2305.01047

\bibitem{Tanimoto}
S. Tanimoto.
\newblock On upper bounds of Manin type.
\newblock{\em Algebra and Number Theory}, 14-3, 731–761, (2020),

\bibitem{Vaughan-Velani-2007}
R.~C. Vaughan and S. Velani.
\newblock Diophantine approximation on planar curves: the convergence theory.
\newblock {\em Invent. Math.}, 166(1):103--124, 2006.


\bibitem{Schmidt-1980}
W.~M. Schmidt.
\newblock {\em Diophantine approximation}.
\newblock Springer Science \& Business Media, Berlin and New York, 1980.


\end{thebibliography}
\providecommand{\bysame}{\leavevmode\hbox to3em{\hrulefill}\thinspace}

\newpage

\end{document}